\documentclass [12pt,a4paper,reqno]{amsart}

\usepackage{mathabx}
\usepackage{amssymb}

\newcommand{\ds}[1]{\ {#1} \ }
\newcommand{\dss}[1]{\quad {#1} \quad }

\input xy
\xyoption{all}
\def\To{\longrightarrow}

\def\Mdot{\cdot_M}

\def\Ndot{\cdot_N}
\def\iso{\tilde{\to}}
\def\isoTo{\tilde{\To}}
\def\into{\hookrightarrow}

\def\zt{\zeta}
\def\00{\{ 0\}}

\def\chU{\widecheck{U}}

\def\chvrp{\widecheck{\vrp}}

\def\tng{\operatorname{t}}

\def\Mix{\operatorname{Mix}}
\def\Eq{\operatorname{Eq}}
\def\Feq{\operatorname{Feq}}
\def\T{\operatorname{T}}
\def\SS{\operatorname{S}}
\def\Is{\operatorname{Is}}
\def\Sis{\operatorname{Sis}}

\def\Tyr{\operatorname{Tyr}}
\def\IGS{\operatorname{IGS}}

\def\vrp{\varphi}
\def\m{m}
\def\t{t}

\def\brv{{\bar v}}

\def\iv{v^{-1}}
\def\itau{\tau^{-1}}

\def\htU{{\widehat U}}

\newcommand\isoto{\xrightarrow{
   \,\smash{\raisebox{-0.45ex}{\ensuremath{\scriptstyle\sim}}}\,}}

\def\igm{\gm^{-1}}

\def\irho{\rho^{-1}}
\def\ial{\al^{-1}}

\def\sig{\sigma}

\def\sm{\setminus}

\def\s{s}

\def\onto{\twoheadrightarrow}
\def\Onto{\; -\hskip-5pt\twoheadrightarrow}

\newcommand{\D}[1]{{\operatorname{D#1}}}
\newcommand{\SC}[1]{{\operatorname{S#1}}}

\def\tng{{\operatorname{t}}}

\def\bbN{\mathbb N}
\def\bbR{\mathbb R}

\newcommand{\etype}[1]{\renewcommand{\labelenumi}{(#1{enumi})}}

\def\STR{\operatorname{STR}}
\def\STROP{\operatorname{STROP}}
\def\STROPm{\STROP_m}

\def\tT{\mathcal T}

\def\tG{\mathcal G}

\def\MFC{\operatorname{MFC}}
\def\MFCt{\operatorname{MFC}_{\tng}}

\def\MFCE{{MFCE}}

\def\dispace{\setlength{\itemsep}{2pt}}
\def\eroman{\etype{\roman} \dispace}
\def\ealph{\etype{\alph} \dispace}

\def\pSkip{\vskip 1.5mm \noindent}

\def\mfM{\mathfrak M}

\def\mfS{\mathfrak S}

\def\mfA{\mathfrak A}

\def\mfq{\mathfrak q}

\def\mfa{\mathfrak a}

\def\al{\alpha}
\def\bt{\beta}
\def\gm{\gamma}

\def\tlx{\tilde{x}}
\def\tlE{\widetilde{E}}

\def\tlN{\widetilde{N}}
\def\tlM{\widetilde{M}}
\def\tlU{\widetilde{U}}
\def\tlE{{\widetilde{E}}}
\def\tlF{{\widetilde{F}}}
\def\tltau{\widetilde{\tau}}
\def\tlal{\widetilde{\al}}
\def\tlvrp{\widetilde{\vrp}}
\def\tlpsi{\widetilde{\psi}}
\def\tlrho{\widetilde{\rho}}
\def\tlgm{\widetilde{\gm}}
\def\tl0{\widetilde{0}}

\def\brx{\bar{x}}
\def\bry{\bar{y}}
\def\brz{\bar{z}}

\newtheorem{thm}{Theorem} [section]
\newtheorem*{thm*}{Theorem}
\newtheorem{cor}[thm]{Corollary}
\newtheorem{lem}[thm]{Lemma}

\newtheorem{prop}[thm]{Proposition}

\newtheorem*{defn*} {Definition}

\newtheorem*{claim*} {Claim}
\newtheorem*{theorem6.6'} {Theorem 6.6$'$}
\newtheorem{acknowledgment*}[thm] {Acknowledgment}
\newtheorem{example}[thm]{Example}
\newtheorem*{example*} {Example}

\newtheorem{addendum}[thm]{Addendum}
\newtheorem{examp}[thm]{Example}

 \newtheorem{rem}[thm]{Remark}
 \newtheorem{rems}[thm]{Remarks}
 \newtheorem{remark}[thm]{Remark}

 \newtheorem*{remark*}{Remark}
 \newtheorem{defn}[thm]{Definition}
\newtheorem{construction}[thm]{Construction}

\newtheorem{schol}[thm]{Scholium}
\newtheorem{notation}[thm]{Notation}
\newtheorem{notations}[thm]{Notations}

\newtheorem*{notation*} {Notation}
\newtheorem*{comment*} {Comment}

 \renewcommand{\sectionmark}[1]{}

\newcommand{\diag}{\operatorname{diag}}

\newcommand{\Cov}{\operatorname{Cov}}

\newcommand{\lm}{\lambda}
\newcommand{\Lm}{\Lambda}

 \newcommand{\id}{\operatorname{id}}

\begin{document}

\title[Supertropical Monoids II] {Supertropical Monoids II: \\[1mm] lifts, transmissions, and equalizers}
\author[Z. Izhakian]{Zur Izhakian}
\address{Institute  of Mathematics,
 University of Aberdeen, AB24 3UE,
Aberdeen,  UK.}
    \email{zzur@abdn.ac.uk}
    \author[M. Knebusch]{Manfred Knebusch}
\address{Department of Mathematics,
NWF-I Mathematik, Universit\"at Regensburg 93040 Regensburg,
Germany} \email{manfred.knebusch@mathematik.uni-regensburg.de}


\subjclass[2010]  {Primary: 13A18, 13F30, 16W60, 16Y60; Secondary:
03G10, 06B23, 12K10,   14T05}

\date{\today}


\keywords{Monoids, supertropical algebra, bipotent semirings,
valuation theory,  supervaluations, transmissions, factorization, lifts, equalizers.}


\begin{abstract}
The category $\STROP$ of commutative semirings, whose morphisms are transmissions, is a full and reflective  subcategory of the category $\STROP_m$ of supertropical monoids.
Equivalence relations on supertropical monoids are constructed easily,  and utilized effectively for supertropical semirings, whereas ideals are too special for semirings. Aiming for tangible factorizations, certain  types of such equivalence relations are constructed and classified explicitly in this paper, followed by a profound study of their characteristic properties with special emphasis on difficulties arising from ghost products of tangible elements.
\end{abstract}

\maketitle

\tableofcontents

\baselineskip 14pt

\numberwithin{equation}{section}

\section*{Introduction}

This paper is a further step in the study of supertropical semirings, combining algebraic and categorical viewpoints, as a sequel of \cite{IKR4}. As in \cite{IKR1}--\cite{IKR5}, we aim for a better understanding of the commutative algebra over these semirings, especially with respect to monoid valuations (written \m-valuations) and their refinement by supervaluations \cite[~\S4]{IKR4}. These valuations generalize the classical valuations, whose targets, which are  ordered abelian groups, are   replaced by ordered monoids \cite[\S2]{IKR1}. Classical valuations, with $\bbR$ as a target, are extensively used in tropical geometry and are at the heart of this theory. Approaching other important valuation targets, such as $\bbN$, has been our initial motivation for developing supervaluations, supported by a rich enough algebraic foundation \cite{I1,I2}, \cite{IR1}--\cite{IzhakianRowen2009Resulatants}, \cite{IKR5}.

Any ordered monoid $M$  gives rise to a bipotent semiring $T(M)$ 
  whose multiplication
is the original monoid operation, and whose addition is the maximum in the given
ordering. For example, the max-plus semiring $T(\bbR)$, the underlying structure of tropical geometry, is obtained by taking the real numbers with the standard ordering and summation. Any bipotent semiring extends to a  supertropical semiring $R$, in which $a + a = \nu(a)$ for any ~$a \in R$, where $\nu: R \to \tG(R)$, called the \textbf{ghost map},  is a projection on a distinguished deal $\tG(R)$, called the \textbf{ghost ideal} of $R$. Equivalently, $a + a = e a$, where $e \in \tG(R) $ is a distinguished idempotent element satisfying $e R = e \tG (R)= \tG (R)$.
This addition,
 which  replaces the rule $a+ a = a$ in max-plus algebra, allows an algebraic capturing of  combinatorial properties, where the ghost ideal takes the role of the zero element in classical algebra. The elements of $\tT(R) = R \sm \tG(R)$ are termed \textbf{tangible} elements. With this setup, any bipotent semiring can be realized as a supertropical semiring having no tangible elements.

In supertropical setting tangible elements are the meaningful elements which essentially  frame the central algebraic and geometric features within the theory. Factorization of tangibles is then a natural issue, rather delicate in this setup, especially  since different $x_1, \dots, x_m \in \tT(R) $ may have the same ghost image $ex_1 = \dots  = e x_m$. Therefore, to characterize factorization of tangibles, the tangible preimage of $e x$ under the ghost map $\nu:R \to \tG(R)$ needs a comprehensive study, including a categorical viewpoint, delivered by valuations and transmissions, as well as the generalization of supertropical semirings to supertropical monoids. Before delving into more  explicit structural details, let us first give a motivating example.

 \begin{example*}
   Let $R$ be a supertropical semiring, where $eR = \{e, c, c^2,\dots  \} \ds{\dot\cup} \{ 0 \}$ is  a totaly ordered monoid with  different $c^i$, and $\tT(R)$ is the free abelian monoid on two generators ~$x,y$. Consider the ghost map $\nu: R \to eR$ that sends $x^i y^j$ to $c^{i+j} =e x^i y^j$, so that the tangible preimage of~ $c^{n}$ consists of all $x^i y^j$ with $i+j =n$.  Identifying two elements $z \sim z'$ with $ez = ez'$ by an equivalence relation $\sim$, which respects the semiring operations, any product ~$z x$ is identified with $z' x$ for every~$x$. Consequently,  factorizations of $z x$ coincide with  $z' x$ when  quotienting by the given  equivalence relation $\sim$. Therefore, since factorizations respect quotienting, they can be explored in terms of maps' factorization which agrees with the ghost maps, where
   one needs a special care of particular pathologies arise in this setting.
 \end{example*}

With this  approach to factorization, let us review the general context of supertropical structures and supervaluations, which links the supertropical theory to classical theory.
Supervaluations refine classical valuations \cite{IKR1}-\cite{IKR3} by replacing their target semirings by supertropical semirings, and provide an enriched algebraic framework to tropical geometry.
An \textbf{\m-valuation}  on a semiring $R$ is
a multiplicative monoid homomorphism $v: R \to M $ to a bipotent
semiring $M$ satisfying $ v(x+y) \ds \leq v(x) + v(y)$, cf.  \cite[\S2]{IKR1}. An {\m-valuation}~$v$ is a \textbf{valuation}, if $M\sm \{ 0 \}$ is a cancellative multiplicative monoid.
For rings, these valuations coincide with the valuations
 defined by Bourbaki \cite{B}, and lead to a mapping of algebraic objects, called \textbf{tropicalization}.
Examples of \m-valuations on rings
  which are not valuations were given  in \cite[\S 1]{IKR1}.

To obtain a categorical framework which comprises supervaluations, in the category $\STROP$ of supertropical semirings (Definition \ref{def:ssmr}), \textbf{transmissions}\footnote{Transmissions are called transmissive maps in \cite[\S5]{IKR1}.}, which are more general than semiring homomorphisms,  appear to be the ``right'' morphisms.
A transmission $\al:R \to S$  is a multiplicative map whose restriction $eR \to eS$ to the ghost ideals is a semiring homomorphism \cite[Definition 1.4]{IKR4}. Transmissions are those maps whose composition $\al \circ v$  with any supervaluation $v$ is again a supervaluation \cite[\S5]{IKR1}. Therefore, $\STROP$ includes bipotent semirings as a subcategory, and provides a richer algebraic setting  for their study.

For a more comprehensive view, the  category $\STROP$ is enlarged to the category $\STROPm$ of \textbf{supertropical monoids}, which contains $\STROP$ as a full subcategory.  A~supertropical monoid $(U, \cdot \, )$ is a pointed  monoid with 
 a distinguished idempotent element $e = e_U$, for which the subset $eU$ carries a total ordering compatible with multiplication, such that $eU$ becomes a semiring by defining addition $+: eU \times eU \to eU$  as the maximum
(Definition \ref{defn:1.1}).
%
%
As in the case of supertropical semirings,  this addition extends to the entire ~$U$ by the use of the multiplication~ $\cdot$ and the idempotent element ~$e$, but distributivity on $(U, \cdot \, , +)$ may fail. When it does not fail, $U$ is a supertropical semiring. A morphism $\al: U \to V$ in $\STROPm$ is a transmission as defined for $\STROP$, which restricts to a homomorphism $eU \to eV$
of bipotent semirings, and obeys the rules TM1-TM5 of transmissions in \cite[Theorem~5.4]{IKR1}.

Since a bipotent semiring $M$ can be viewed as a supertropical
semiring having only ghost elements, where $e =
1_M,$  the category $\STROP_m /M$ of supertropical monoids over~
$M$ may be viewed as the category of supertropical monoids $U$
having a fixed ghost ideal $eU = M$, whose morphisms
 are  transmissions $\al: U \to V$ with $\al(x) = x$ for all
$x \in M.$ The surjective transmissions over $M$ are called
\textbf{fiber contractions}, as in the case of
supertropical semirings \cite[\S6]{IKR1}. If $\al:U \onto V$ is a fiber contraction and $U$ is a
supertropical semiring, then $V$ is also  a supertropical semiring
 \cite[Theorem 1.6]{IKR4}, and~$\al$ is a semiring homomorphism
\cite[Proposition 5.10.iii]{IKR1}.

 For every supertropical monoid $U$ there exists a fiber
contraction  $\sig_U: U \to \htU$,  where~ $\htU$ is the \textbf{supertropical semiring
associated} to $U$, such that every fiber contraction $\al:~U \onto
V$ factors (uniquely) through $\sig_U$, i.e., $\al = \bt \circ \sig_U$, where
 $\bt:\htU \to V$ is  a fiber contraction. Namely, $\STROP /M$ is a full reflective subcategory of $\STROP_m
/M$~\cite[~p9]{F}, \cite[1.813]{FS}.

Universal problems appearing in
$\STROP/ M$ are generalized to $\STROPm /M$ in an obvious way and can be solved, often more easily.
Such solutions are delivered to $\STROP / M $ by \textbf{reflections} $\sig_U: U \to \htU$.
%
%
%
%
This approach pertains in particular to \m-valuations  $v: R \to M$ on a ring $R$, and to  supervaluations $\vrp:R \to S$ with $S$ a supertropical semiring over $e S = M$, as defined in \cite[\S4]{IKR1}.  (They are also defined for $R$ a semiring.)
Such a supervaluation applies to the coefficients of a (Laurent) polynomial $f(\lm) \in R[\lm]$ in a set $\lm$ of $n$
variables, and gives a polynomial $f^\vrp(\lm)$ over $S$.  This view  helps to analyze   supertropical root sets and tangible
components of polynomials $g(\lm) \in F[\lm]$, 
obtained from $f^\vrp(\lm) \in S[\lm]$ by
passing from~ $S$ to various supertropical semifields $F$ \cite[\S5 and~ \S7]{IR1}. To  this end,
 one needs a good control on the set
$ \{ a \in R \ds |  \vrp(a) \in M \} \ds  =  \{ a \in R \ds |  \vrp(a) = v(a) \},$
which plays a central role in \S\ref{sec:1} and $\S\ref{sec:2}$ below.

Given an \m-supervaluation $\vrp: R \to U$ covering an \m-valuation  $v: R \to M$, in \S\ref{sec:1}
we construct a tangible $\text{\m-supervaluation}$ $\tlvrp: R \to
\tlU$ which is minimal such that  $\tlvrp \geq \vrp$ (Theorem
\ref{thm:1.10}).
 In \S\ref{sec:2} we then classify
the \m-supervaluations $\psi$ satisfying  $\vrp \leq \psi \leq \tlvrp,$
called the \textbf{partial tangible lifts} of $\vrp.$ They  are
uniquely determined by their \textbf{ghost value sets}
$$ G(\psi) := \psi (R) \cap M,$$
cf.~Theorem \ref{thm:2.4}. These sets are ideals of the semiring $M,$
and all ideals $\mfa \subset G(\vrp)$ occur in this way (Theorem
\ref{thm:2.7}). Unfortunately, the ghost value set $G(\psi)$  does
not control the set $\{ a \in R \ds | \psi(a) \in M\} $
completely.  We  only able to state that this set  is contained in the preimage
$\iv(G(\psi))$.
If $\vrp$ is a supervaluation,  then $\chvrp := (\tlvrp)^\wedge$
is the supervaluation, which is a partial tangible lift of $\vrp$
having smallest ghost value set.

In \S\ref{sec:3} we undertake a fine analysis of surjective transmissions $\al: U' \twoheadrightarrow U$ for supertropical monoids $U', U$. Such transmissions can be approached via TE-relations, i.e., equivalence relations on $U'$, compatible with its monoid structure in the appropriate sense \cite[~Definition~ 1.7]{IKR4}.
Namely, given a transmission  $\al: U' \twoheadrightarrow  U$, it induces the equivalence relation  $E = E(\al)$ on $U'$ which identifies all elements $x' \in U'$ having the same image under $\al$.  Then $\al$ factorizes as $\al = \rho \circ \pi_E$, where~
$\rho:U'/E \isoto U$ is an isomorphism  and
 $\pi_E: U' \onto  U'/E$ is the canonical surjection  sending each $x' \in U'$ to its equivalent class~ $[x']_E$.

Dividing out the zero kernel of a transmission $\al: U'  \to U$ \cite[\S1]{IKR4}, without loss of generality we assume that
$\al^{-1}(0) = \{ 0' \}$. A transmission $\al$ is called \textbf{tangible} (Definition~ \ref{defn:1.1}), if $\al(\tT(U')) \subset \tT(U) \cup \00$. It is called \textbf{mixing}, if for any different elements $x'_1, x'_2 \in U'$ with $\al(x'_1) = \al(x_2')$ there exists a mixing element $z' \in U'$, might interchanging $x_1'$ and $x_2'$, such that $x_1' z \in \tT(U')$ and $x_2'z \in e U'$. We prove  that every surjective transmission $\al : U' \onto U$ with trivial kernel factorizes as
\begin{equation}\label{eq:0.1}
  \al = \al_m \circ \al_t,
\end{equation}
where $\al_t: U' \onto V$ is a tangible transmission, and $\al_m : V \onto U$ is a mixing transmission which covers the identity of $eU$, and thus is a multiplicative fiber contraction of $V$ over $eV = eU$.

 Choosing as small as possible $V$ in the appropriate sense, the factorization \eqref{eq:0.1} is unique (Theorem \ref{thm:3.10}.(ii)), called the \textbf{$(t,m)$-factorization} of $\al$. The factor $\al_m$ can be replaced by the surjection  $\pi_E: V \onto V/E$ of a suitable MFCE-relation  (=multiplicative fiber conserving equivalence) $E$  on $eV = eU$. This indicates that finding $(t,m)$-factorizations is essentially a matter of MFCE-relations, and so  can be  approached explicitly by the methods developed in \cite{IKR1}--\cite{IKR4}.
In \S\ref{sec:4} we determine the $(t,m)$-factorizations of the  composite $\bt \circ \al$ of two surjective transmissions $\al$ and $\bt$ having trivial zero kernel.

The idea beyond $\S\ref{sec:3}$ and $\S\ref{sec:4}$ is to study the ``fate'' of a given (tangible)  family $\{x_1, \dots, x_r \}$ in a  supertropical monoid $U$,  satisfying $ex_1= \cdots = ex_r$, due to multiplying by some element $u \in U$. This means to decide whether $x_i u \in \tT(U)$ or $x_i u \in e U$, i.e., the products becomes a ghost or zero.  The latter  case  with $u \in \tT(U)$ is of  more interest. We follow this idea in ~\S\ref{sec:5}--\S\ref{sec:7},  relying on hierarchy of elements which is phrased by a patriarchal  terminology.

An element $z \in \tT(U)$ is called a \textbf{son of $x \in \tT(U)$ over $c = ex$}, if $ez = c$ and there exists $u \in \tT(U)$ such that $z = xu$. Taking a patriarchal  viewpoint,  $z$ is regarded as the outcome~ $xu$ of pairing $x$ to some $u \in \tT(U)$ such that $exu = c$, without specifying explicitly the ``mother'' element $u$ of $z$. In other words, $x$ and $z$ have the same ghost image $c$, where $z = xu$ for some unspecified tangible element $u$.
An element $x$ is called a \textbf{tyrant} (or said to be tyrannic), if it has at most one son over $c \in eU$.
It is said to be \textbf{isolated}, if it has no son~ $z \neq x $ with $ex = ez$. The latter property is weaker than tyrannic, but to our opinion deserves similar attention.

For any tyrant $x$ in a supertropical monoid $U$ the set $\mfa = eU \cup xU$ is obviously an ideal of $U$ containing $M =e U$. This ideal is small in the sense, that each fiber of the restricted ghost  map $\nu_U | \mfa$ has at most two elements, $\mfa_c = \{c,z \}$ if $x$ has a son $z$ over $c$, otherwise $\mfa_c = \{c\}$. Conversely, given an ideal $\mfa \supset M$ with $|\mfa_c| \leq 2$ for every $c \in M$, then $\nu_U: U \to M$ has a section $s: M \to U$ for which $\mfa = M \cup s(M)$. As explicated in \S\ref{sec:5}, it follows that the MFCE-relations, for which every class contains at most one tangible element, are the compressions of these ideals $M \cup s(M)$, cf. \cite[~Definition~ 2.5]{IKR4}.

The \textbf{equalizer} $\Eq(S)$ of a subset $S$ of $U_a = \{ x \in U \ds | ex = a\} $ is defined as the finest equivalence relation $E$ on $U$ with $s \sim_E t$ for all $s,t \in S$. More generally, given an arbitrary subset $S$ of $U$, the \textbf{fiberwise equalizer} $\Feq(S)$ of $U$ is the finest equivalence relation $E$ on $U$ for which $s \sim_E t$ for any $s,t \in S$ with  $es = et$. $\Eq(S)$ and $\Feq(S)$  are obtained in \S\ref{sec:6} by explicit construction, using the so-called ``$S$-pathes'' in $U$. These pathes serve as our main technical tool in \S\ref{sec:7}.

The \emph{universal MFCE-relation $\T(x)$ turning a given $x \in \tT(U)$ into a tyrant} is thoroughly  studied  in ~\S\ref{sec:7}. This means that the relation $\T(x)$ is the finest MFCE-relation $E$ on $U$ for which~ $[x]_E$ is tyrannic on ~$U/E$, provided that such relation $E$ exists on $U$, and otherwise $[x]_{\T(x)}$ is ghost in $U/\T(x)$. This relation $\T(x)$ is the fiberwise equalizer $\Feq(\SS(x))$ of the set $\SS(x) = (Ux) \cap \tT(U)$ of sons of $x$.
Using the results of \S\ref{sec:6}, we  gain an explicit criterion when $[x]_{\T(x)}$ is a tyrant and when is not (Theorem \ref{thm:7.6}).

In an analogous way we obtain the universal MFCE-relation $\Is(x)$ which isolates $x$ in its fiber $U_{ex}$, as well as the universal MFCE-relation $\Sis(x)$ which isolates every son~ $z$ of $x$ in its fiber~ $U_{ez}$, together with criteria for the existence of such MFCE-relations (Theorems~ \ref{thm:7.9} and~ \ref{thm:7.16}). It turns out, perhaps astonishing at first glance, that under a mild cancellation hypothesis on $eU$ these MFCE-relations coincide (Theorem \ref{thm:7.18}) and  $\Is(x) =\Sis(x)$.

The  results in \S\ref{sec:3}--\S\ref{sec:7} may be viewed as instances of a supertropical divisibility theory, performed upon  TE-relations instead of ideals, in particular prime ideals, in classical commutative algebra. In the same spirit, prime and radical ideals can be replaced by special equivalence relations to develop a systematic theory of  commutative $\nu$-algebra~ \cite{I2}, which generalizes supertropical algebra.

\section{Lifting ghosts to tangibles}\label{sec:1}
We recall our underlying structures.

\begin{defn}\label{def:ssmr}
A \textbf{supertropical semiring} is  a semiring $R$ where $e:=1+1$ is  idempotent
(i.e., $e+e = e$) such that,   for
all~ $a,b\in R$,  $a+b\in\{a,b\}$ whenever $ea \ne eb$ and  $a+b=ea$ otherwise. This implies
 $ e a = 0 \ds \Rightarrow a = 0.$
The  target of the \textbf{ghost map}
$\nu: a\mapsto ea$ is the \textbf{ghost ideal} $eR$ of $R$, which is a bipotent semiring (with unit element ~ $e$),  i.e., $a+b$
is either $a$ or $b$, for any $a,b\in eR$.
The semiring $eR$ is totally ordered  by the rule
\begin{equation}\label{eq:1.1}
 a\le b \dss\Leftrightarrow a+b=b.
\end{equation}
$\tT(R):=R\sm(eR)$ is the set of \textbf{tangible} elements of $R$, and $\tG(R) := eR \sm \{ 0 \}$ is the set of \textbf{nonzero ghost elements}.
 The
zero $0 = e0$ is regarded mainly as a ghost.
\end{defn}
More generally, we have the following structure.
\begin{defn}\label{defn:1.1}
 A \textbf{supertropical monoid} $U$ is
an abelian monoid $(U, \cdot \, )$   with  an
absorbing element $0 := 0_U$, i.e., $0 \cdot x = 0$ for every $x
\in U$, and a distinguished idempotent $e := e_U$ such
that
$$ \forall x \in U: \quad ex =0 \ \Rightarrow \ x = 0.$$
In addition, the submonoid $M := eU$ of $U$ is equipped  with a total ordering, compatible with multiplication \cite[Definition 1.1]{IKR4}, which is again determined by rule \eqref{eq:1.1}. The map $\nu_U: U \to M$, $x \mapsto ex$, is a monoid homomorphism, called the \textbf{ghost map} of $U$. Tangible elements $\tT := \tT(U)$ and ghost elements $\tG := \tG(U)$ are defined exactly as in Definition~ \ref{def:ssmr}.

A supertropical monoid $U$ is called \textbf{unfolded}, if the set
$\tT(U)_0 := \tT(U) \cup ~\{0 \}$ is closed under multiplication.
\end{defn}

If $U$ is unfolded, then $N:= \tT(U)_0$ is a multiplicative monoid with absorbing  element~$0$. Furthermore,  $M := eU$ is a
totally ordered monoid with absorbing  element $0$, and the
restriction
$$ \rho := \nu_U |N : N \ds \To M$$
is a monoid homomorphism with $\irho (0) = \{ 0 \}$. Observing
 that $e_U = 1_M = \rho(1_N)$, we see that the supertropical
monoid $U$ is  completely determined by the triple $(N,M, \rho)$.
This paves a way to constructing  all unfolded supertropical
monoids up to isomorphism.
\begin{construction} \label{constr:1.2} Given a
totally ordered monoid $M$ with absorbing element $0_M \leq x$ for
all $x \in M$, i.e., a bipotent semiring $M$, let $N$ be an (always
commutative) monoid with  absorbing element  $0_N$, together with a
multiplicative map $ \rho : N \to M$ with $\rho (1_N) =
1_M$, $\irho(0_M) = \{ 0 _N\} $. We define an unfolded
supertropical monoid $U$ to be the disjoint
union of $M \sm \{ 0_M\}$, $N \sm \{ 0_N\}$ and  a new element
$0$, identified as $0 = 0_M = 0_N $. We  write
$$U = M \cup N, \quad \text{where } M \cap N = \{ 0 \}.$$

The multiplication on $U$ is given by the rules, in obvious
notation,
$$  x \cdot y = \left\{
\begin{array}{lll}
  x \Ndot y & \text{if} &  x,y \in N, \\[1mm]
  \rho(x) \Mdot y & \text{if} &  x \in N, y \in M, \\[1mm]
  x \Mdot \rho(y) & \text{if} &  x \in M, y \in N, \\[1mm]
  x \Mdot y & \text{if} &  x,y \in M.  \\
\end{array}
\right.
$$
It is easy to verify that $(U, \cdot \,)$ is a (commutative)
monoid with $1_U = 1_N$ and absorbing element $0$. Let $e := 1_M$, then $eU = M$ and $\rho(x) = ex$ for $x \in M $. Furthermore,    $ex =
0$ iff $x=0$ for any $x \in U$, since $\irho(0) = \{ 0\}$. Thus
$(U, \cdot \, , e)$, together with the given ordering on $M = eU$,
is a supertropical monoid,  clearly unfolded. We denote the
supertropical monoid $U$ by $\STR(N,M, \rho)$.\footnote{The notation $\STR(N,M, \rho)$ differs slightly from the
notation $\STR(\tT, \tG, v)$ in \cite[Construction ~3.16]{IKR1}, causes no confusion
regarding the ambient context.}
\end{construction}

This construction generalizes the construction of supertropical
domains \cite{IKR1} (loc. cit. Construction 3.16). There,
to obtain all supertropical predomains up to
isomorphism, it was assumed that $N \sm \{ 0 \}$ and $M \sm \{ 0 \}$ are closed under
multiplication, and that the monoid $M \sm\{ 0 \}$ is
cancellative. Omitting only the cancellation hypothesis would
give us a class of supertropical monoids not broad enough for our
work below.

\pSkip

We add a description of the transmissions  between two unfolded
supertropical monoids.

\begin{prop}\label{prop:1.3} Assume that $U' = \STR(N', M', \rho')$
and $U = \STR(N, M, \rho)$ are unfolded supertropical monoids.
\begin{enumerate} \eroman
    \item If $\lm: N' \to N$ is a monoid homomorphism with $\lm(0) =
    0$,  $\mu: M' \to M$ is a semiring homomorphism, and  $\rho' \lm = \mu
    \rho$, then the well-defined map
    $$ \STR(\lm,\mu) : U' = N' \cup M' \dss \To  U = N \cup M,$$
    sending  $x' \in N'$ to $\lm(x')$ and $y' \in M'$ to
    $\mu(y')$,
    is a tangible transmission \cite[~Definition~2.3]{IKR4}.

    \item In this way we obtain all tangible transmissions from
    $U'$  to $U$.
\end{enumerate}
\end{prop}
\begin{proof} (i): A straightforward check.
\pSkip
(ii): Obvious.
\end{proof}

Given an \m-valuation $v: R \to M$
with support $\iv(0) = \mfq$, the supertropical semiring~ $U^0(v)$
appearing  in \cite[Theorem 7.4]{IKR4} may be viewed as an instance
of Construction ~\ref{constr:1.2}, as follows.

\begin{example}\label{exm:1.4} Let $E$ denote the equivalence
relation on $R$ having the equivalence classes $[0]_E = \mfq$ and $[x]_E
= \{ x \} $ for $x \in R\sm \mfq$. $E$ is multiplicative, and  hence
 $R / E$ is  a monoid with absorbing element $[0]_E = 0$,
 identified  with the subset $(R \sm \mfq) \cup \{ 0 \}$ in
the obvious~way.  The map $v: R\to M$ induces a monoid
homomorphism $\brv: R/E \to M$ given by $\brv(x) = v(x)$ for $x
\in R \sm \mfq$, $\brv(0) = 0 $. Thus  $\brv^{-1} (0) = \{ 0 \}$
and
$$U^0(v) = \STR(R / E, M , \brv).$$
\end{example}

We next aim for ``unfolding'' methods of an arbitrary supertropical
monoid $U$. By this we roughly mean a fiber contraction $\tau:
\tlU \to U$, where $\tlU$ is an unfolded supertropical monoid and the
fibers $\itau(x)$, $x\in U$, are as small as possible. More precisely
we decree

\begin{defn}\label{defn:1.4}
Let $M := eU$, and let $N$ be a submonoid of $(U, \cdot \;)$ which
contains the set $\tT(U)_0$. An \textbf{unfolding of $U$ along
$N$} is a fiber contraction $\tau: \tlU \to U$ over $M$ (in
particular $e \tlU = M$), such that
$$  \itau(x)  = \left\{
\begin{array}{lll}
  \{x, \tlx \} & \text{if} &  x \in M \cap N, \\[1mm]
  \{ \tlx \} & \text{if} &  x \in N \sm M, \\[1mm]
    \{x \} & \text{if} &  x \in M \sm N, \\
\end{array}
\right.
$$
where $\tlx \in \tT(U)_0$. For any $x \in N$,  we call $\tlx$ the
\textbf{tangible lift} of $x$ (with respect to $N$).
\end{defn}

Notice that this forces $\tau(\tT(\tlU)_0)= N$, and  moreover
for any $x \in N$ the tangible fiber~$\tlx$ is the unique element
of $\tT(\tlU)_0$ with $\tau(\tlx) = x$, hence $\tT(\tlU)_0 = \{
\tlx \ds | x\in N \}$.
Thus, if $\tau:\tlU \to U$ is an unfolding along  $N$, then the restriction  $\tT(\tlU)_0 \to N$, $\tlx \mapsto x$, of $\tau$   is a monoid isomorphism, and $\tau$ itself
is an ideal compression  with ghost kernel
 ${(M \cap N)}^{\sim} \cup M$ (cf. \cite[Definition 1.5]{IKR4}), where ${(M
\cap N) }^{\sim}:= \{ \tlx \ds | x \in M \cap N \}$.


\begin{thm}\label{thm:1.5} $ $
\begin{enumerate} \eroman
    \item Given a pair $(U,N)$ consisting of a supertropical
    monoid $U$ and a multiplicative submonoid $N \supset
    \tT(U)_0$, there exists an unfolding $\tau: \tlU \to U$ of $U$
    along $N$.

    \item  If $\tau': \tlU' \to U$ is a second unfolding of $U$
    along $N$, then there exists a unique isomorphism of
    supertropical monoids $\al : \tlU \ds \iso \tlU'$ with $\tau' \circ \al =
    \tau$.
\end{enumerate}

\end{thm}

\begin{proof} (i) \emph{Existence:} Since $M$ is an ideal of $U$,
the set $M \cap N$ is a monoid ideal of $N$. We have $U = N \cup
M$, since  $N \supset
    \tT(U)_0$.  Let $\rho: N \to M$ denote the restriction of
    $\nu_U$ to ~$N$. It is a monoid homomorphism with $\irho(0) = \{ 0
    \}$.

    Let $\tlN$ denote a copy of the monoid $N$ with copying
    isomorphism $x \mapsto \tlx$ ($x \in N$), and let $\tlrho: \tlN \to
    M$ denote the monoid homomorphism from $\tlN$ to $M$
    corresponding to $\rho: N \to M$. Thus $\tlrho (\tlx) = \rho(x) =
    ex$ for $x\in N$. Define the unfolded supertropical monoid
    $$ \tlU := \STR(\tlN, M, \tlrho) = \tlN \cup M.$$
In $\tlU$ we have $\tl0_U = 0 $ and $\tlN \cap M = \{ 0 \}$.
Further $\tT(\tlU)_0 = \tlN$ and $e \tlU = eU = M$.

We obtain a well-defined surjective map $\tau: \tlU \to U$ by
setting  $\tau(\tlx) := x$ for $x\in N$, $\tau(y) := y$ for $y\in
M$. As easily checked, $\tau$ is multiplicative, sending
$0$ to $0$, $1 \in \tT(\tlU)$ to $1 \in N$, which restricts to the
identity on $M$. Thus $\tau$ is a fiber contraction
\cite[Definition~2.1]{IKR4}. The fibers of $\tau $ are as
indicated in Definition~\ref{defn:1.4}; hence $\tau$ is an
unfolding of $U$ along $N$. \pSkip
(ii) \emph{Uniqueness}: Let $\tltau: \tlU \to U$ and $\tltau':
\tlU' \to U$ be unfoldings of  $U $ along $N$ with tangible lifts
$x \mapsto \tlx$ and $x \mapsto \tlx'$ respectively. Without loss
of generality we assume that $\tlU  = \STR(\tlN,M, \tlrho)$ and
$\tlU' = \STR(\tlN',M, \tlrho')$ with tangible lifts $x \mapsto
\tlx$ and $x \mapsto \tlx'$ ($x \in N$). Then $\tlrho(\tlx) =
\tlrho'(\tlx') = ex$ for every $x \in N$.  The map $\lm: \tlN \to
\tlN'$, given by $\lm (\tlx) = \tlx'$ for $x \in N$, is a monoid
isomorphism with $\tlrho' \circ \lm = \id_M \circ \tlrho$. Thus we
have  a well defined transmission (cf. Proposition
\ref{prop:1.3})
$$ \al := \STR(\lm, \id_M): \tlU \ds \To \tlU'.$$   $\al$ is an isomorphism over $U$, i.e., an
isomorphism with $\tau' \circ \al = \tau$, clearly the only one.
\end{proof}

\begin{notation}\label{notation:1.6}
We call the map $\tau: \tlU \to U$ constructed in part (i) of the
proof of Theorem~\ref{thm:1.5} \textbf{``the'' unfolding of $U$
along $N$} and, when necessary, write it  more  precisely as
$$ \tau_{U,N}: \tlU(N) \To U.$$   Sometimes we
abusively denote any unfolding of $U$ along $N$ in this way,
justified by Theorem \ref{thm:1.5}.(ii).
\end{notation}

\begin{example}\label{exam:1.7} In  the special case
that $U = eU =M$, $N$ can be any submonoid of $M$ containing
$0$. Then  $\tlM(N) = \tlN \cup M$ with $\tlN \cap M = \{0\}$, and
$$\tlM(N) \cong \STR(N,M, i)$$ with $i : N \into M$ the inclusion
mapping. For every $x \in N$ there exists a unique tangible
element $\tlx$ in $\tlM(N)$ with $e \tlx =x$, while for $x \in M
\sm N$ there exists no such element.
\end{example}

\begin{thm}\label{thm:1.8} Assume that $\al : U' \to U$ is a
transmission of  supertropical monoids \cite[Definition 1.4]{IKR4}, and that $N' \supset
\tT(U')_0$, $N \supset \tT(U)_0$ are submonoids of $U'$ and $U$
with $\al(N') \subset N$. Then there exists  a unique tangible
transmission
$$ \tlal:= \tlal_{N',N} : \tlU' (N') \ds \To \tlU(N),$$
called the \textbf{tangible unfolding of $\al$ along $N'$ and
$N$}, such that the diagram
$$
\xymatrix{   \tlU'(N')   \ar @{>}[d]^{\tau_{U',N'}}  \ar
@{>}[rr]^{\tlal } &   & \tlU(N)
\ar @{>}[d]^{\tau_{U,N}}   \\
U'  \ar @{>}[rr]_{\al}  &  & U  }
$$
commutes.
\end{thm}
\begin{proof} Let $M' := eU'$, $M := eU$, and let $\rho': N' \to
M$, $\rho: N \to M$ denote the monoid homomorphisms  obtained from
$\nu_{U'}$ and $\nu_U$ by restriction to $N'$ and $N$. Then
$$\tlU'(N') = \STR(N',M', \rho'), \qquad \tlU(N) = \STR(N,M,\rho). $$
The map $\al$ restricts to monoid homomorphisms $\lm: N' \ds \to
N$ and $\gm: M' \to M$ with $\lm(0) = 0$, $\gm(0) = 0 $, where  $\gm$ is
order preserving. So  $\gm \circ \nu_{U'} = \nu_U \circ \al$,
hence $\gm \circ \rho' = \rho \circ \lm$, and we obtain the tangible
transmission
$$ \tlal:= \STR(\lm, \gm): \tlU'(N') \ds \To \tlU(N).$$
Clearly  $\tau_{U,N} \circ \tlal = \al \circ
\tau_{U',N'}$. Since any tangible transmission from $\tlU'(N')$ to
$\tlU(N)$ maps~ $\tlN'$ to $\tlN$ and $M'$ to $M$, it is evident
that $\tlal$ is the only such map.
\end{proof}

\begin{cor}\label{cor:1.9} Assume that $\al: U' \to U$ is a tangibly
surjective
transmission of  supertropical monoids, i.e., $\tT(U) \subset \al(\tT(U'))$, and $U'$ is unfolded. Let $N:= \al (\tT(U')_0)$, which is a
submonoid of $U$ containing $\tT(U)_0$.
\begin{enumerate} \eroman
    \item There exists a unique tangible transmission
    $$\tlal: U' \ds \To \tlU(N), $$
    called the \textbf{tangible lift} of $\al$, such that $\tau_{U,N} \circ \tlal = \al.$
    \item If $x' \in U'$, then
    $$  \tlal(x')  = \left\{
\begin{array}{lll}
  \widetilde{\al(x')} & \text{if} &  x' \in \tT(U')_0, \\[1mm]
  \al(x') & \text{if} &  x' \in eU'. \\
\end{array}
\right.
$$
\end{enumerate}
\end{cor}

\begin{proof} (i): Apply Theorem~\ref{thm:1.8} with $N' :=
\tT(U')_0$, and observe that $\tlU'(N') = U'$, since ~$U'$ is
unfolded.
 \pSkip
(ii): Obvious, since $\tau_{U,N} (\tlal(x')) = \al (x')$ and
$\tlal(x') \in \tT(\tlU)_0$ iff $x' \in \tT(U')_0$.
\end{proof}
\pSkip

We are   ready to construct ``tangible lifts'' of
\m-supervaluations.

\begin{thm}\label{thm:1.10}
Assume that $\vrp: R \to U$ is a tangibly
surjective  \m-supervaluation, i.e., $\tT(U) \subset \vrp(R)$ \{e.g. $\vrp$
is surjective; $U = \vrp(R) \cup e \vrp(R)$\}. Let $N := \vrp(R)$. It is a submonoid of $U$ containing $\tT(U)$.
\begin{enumerate} \eroman
    \item The map
    $$ \tlvrp : R \To \tlU(N), \qquad a \mapsto \widetilde{\vrp(a)},$$
    with $\widetilde{\vrp(a)}$ denoting the tangible lift of
    $\vrp(a)$ w.r.t. $N$, is a tangible \m-supervaluation of $\vrp$,
    called the \textbf{tangible lift} of $\vrp$ \cite[Definition 2.3]{IKR4}.

    \item If $\vrp' : R \to U'$ is a tangible \m-supervaluation
    dominating $\vrp$, then $\vrp'$ dominates $\tlvrp$.
\end{enumerate}
\end{thm}
\begin{proof} (i): $\tlvrp$ is multiplicative, $\tlvrp(0) = 0$,  $\tlvrp(1) =
1$, and $e \tlvrp = e \vrp$ is an \m-valuation. Thus ~$\tlvrp$ is
an \m-supervaluation. By construction, $\tlvrp$ is tangible.
\pSkip
(ii): We may assume that the \m-supervaluation $\vrp': R \to U'$
is surjective, and hence   $\vrp'(R) \supset \tT'(U)_0$. Since
$\vrp'$ is tangible, this forces $\vrp'(R) =\tT'(U)_0$. Thus
$\tT'(U)_0$ is a submonoid of $U'$, i.e., $U'$ is unfolded. Since
$\vrp'$ dominates $\vrp$, there exists a transmission $\al: U' \to
U$ with $\vrp = \al \circ \vrp'$, such that  $$\al(\tT'(U)_0) = \al
(\vrp'(R)) = \vrp (R) = N.$$ Thus we have the  tangible lift of $\al$,
$$\tlal: U' \ds \To \tlU(N).$$ So, for any $a\in R,$
$$ \tlal(\vrp'(a)) = [\al (\vrp'(a))]^\sim  = \widetilde{\vrp(a)} = \tlvrp(a).$$
Thus $\tlvrp = \tlal \circ \vrp',$ which proves that $\vrp'$
dominates $\tlvrp.$
\end{proof}

\begin{addendum}\label{adden:1.11} As the proof  has shown, if the
\m-valuation $\vrp'$ is surjective, then $U'$ is unfolded, and the
transmission \cite[Definition~5.3]{IKR1} $$\al_{\tlvrp,\vrp'}: U' \ds \To \tlU(N)$$  is the  tangible lift of $\al_{\vrp,
\vrp'}: U' \to U$.
\end{addendum}

\begin{cor}\label{cor:1.12} If $\vrp$, $\psi$ are
\m-supervaluations covering $v$ and $\vrp \leq \psi$, then $\tlvrp
\leq \tlpsi$.
\end{cor}
\begin{proof} We have  $\vrp \leq \psi \leq \tlpsi$, and from
Theorem~\ref{thm:1.10}.(ii) it follows  that $\tlvrp \leq \tlpsi$.
\end{proof}

\section{The partial tangible lifts of an
$\m$-supervaluation}\label{sec:2}

In what follows $v: R \to M$ is a fixed \m-valuation and
$\vrp: R \to U$ is a tangible surjective \m-supervaluation
covering $v$. (Often $v$ and $\vrp$ will both be surjective.)
 The tangible  lift $\tlvrp: R \to
\tlU$ (cf. Theorem \ref{thm:1.10}) was introduced  in  \S\ref{sec:1}, and now we strive for an explicit
description of the \m-supervaluations $\psi$ covering $v$ with
$\vrp \leq \psi \leq \tlvrp$.


\begin{defn}\label{defn:2.1}Given an \m-supervaluation $\psi: R \to V$ covering $v: R \to M$,
we call
$$ G(\psi) := \psi(R) \cap M = \{ \psi(a) \ds | a\in R, \ \psi(a)= v(a) \}$$
the \textbf{ghost value set} of $\psi$. \{Notice that $eV = M$.\}
\end{defn}

\begin{lem}\label{lem:2.2}
Let $\psi_1, \psi_2$ be \m-supervaluations covering $v$. If
$\psi_1 \geq \psi_2$, then $G(\psi_1) \subset ~G(\psi_2)$. If
$\psi_1 \sim \psi_2$, then $G(\psi_1) = G( \psi_2)$.
\end{lem}

\begin{proof} Let $a\in R$. If $\psi_1 \geq \psi_2$, then $\psi_1(a) \in
M$ implies that $\psi_2(a) \in M$, due to dominance, condition $\D 3$ in \cite[Definition~7.2]{IKR4}. Thus,
 $\psi_1(a) \in M$ iff $\psi_2(a)
\in M$, for $\psi_1 \sim \psi_2$.
\end{proof}

\begin{lem}\label{lem:2.3} Assume that the \m-valuation $v: R \to
M$ is surjective. Then the ghost value set $G(\psi)$ of any
\m-supervaluation $\psi$ covering $v$ is an ideal of the semiring
$M$.
\end{lem}

\begin{proof}
For  $x \in G(\psi)$, $y \in M$ there exist $a,b \in R$ with
$\psi(a) =x$, $e \psi(b) = y$, implying that
$$ xy = e \psi(a) \psi(b) = \psi(a) \psi(b) = \psi(a b).$$
Thus $xy \in \psi(R) \cap M = G(\psi)$, which proves that $G(\psi)
\cdot M \subset G(\psi)$. Since $M$ is bipotent, $G(\psi)$ is also
closed under addition.
\end{proof}

\begin{thm}\label{thm:2.4} Assume that $\vrp: R \to U$ is an
\m-supervaluation  covering $v: R \to M$, and that $\psi_1,
\psi_2$ are \m-supervaluations covering $v$ with $$ \vrp \leq
\psi_1 \leq \tlvrp, \quad  \vrp \leq \psi_2 \leq \tlvrp.$$
\begin{enumerate} \eroman
    \item $\psi_1 \geq \psi_2 \dss \Leftrightarrow G(\psi_1) \subset
    G(\psi_2)$;
    \item $\psi_1 \sim \psi_2 \dss \Leftrightarrow G(\psi_1) =
    G(\psi_2)$.
\end{enumerate}
\end{thm}

\begin{proof} Without loss of generality, assume that $\vrp$ is
surjective. Then also the \m-super-valuations $\psi_1,$ $\psi_2,$
$ \tlvrp$ are surjective and,  by Corollary \ref{cor:1.12}, the tangible
lifts $\tlpsi_1$ and $\tlpsi_2$ are both equivalent to~$\tlvrp$.

Again,  without loss of generality,  assume that $\vrp =
\tlvrp / E := \pi_E \circ \tlvrp$ with $E$ an \MFCE-relation on
$\tlU$, and that $\psi_i = \tlvrp / E_i$ with an \MFCE-relation
$E_i$ ($i =1,2$). Let us describe these relations $E, E_1, E_2$
explicitly. We have $\tlU = \tlN \cup M$, with
$$ \tlN := \tT(\tlU)_0 = \tlvrp(R), \qquad \tlN \cap M = \{ 0\},$$
furthermore $U = N \cup M$ with
$$ N:= \vrp (R), \qquad N\cap M = G(\vrp),$$
and a copying isomorphism $$s: N \ds \isoTo \tlN$$ of
monoids (new notation!), which sends each $x \in N$ to its
tangible lift $\tlx$, as explained in~ \S\ref{sec:1} (Definition
\ref{defn:1.4}, Proof of Theorem \ref{thm:1.5}.i). Notice that
$es(x)=x $ for $x \in N \cap M = G(\vrp).$

The relation $E$ has the 2-point equivalence classes $\{x, s(x)
\}$ with $x$ running through~ $G(\vrp)$, while all other
$E$-equivalence classes are one-point sets. Analogously, $E_i$ has
the 2-point set equivalence classes $\{x, s(x) \}$ with $x$
running through $G(\psi_i) \subset G(\vrp)$, while again all other
$E$-equivalence classes are one-point sets. Thus it is obvious
that\footnote{As in \cite{IKR1} we view an equivalence relation on
a set $X$ as a subset of $X \times X$ in the usual way.} $E_1
\subset E_2$ iff $G(\psi_1) \subset G(\psi_2)$. But $E_1 \subset
E_2$ means that $\psi_1 \geq \psi_2$. This gives  claim (i), and
claim (ii) follows.
\end{proof}

\begin{defn}\label{defn:2.5} We call the monoid isomorphism
$$ s: \vrp(R) \ds \To \tT(\tlU)_0 = \tlvrp(R),$$
i.e., the copying isomorphism $s: N \iso \tlN$ from the
proof of Theorem \ref{thm:2.4}, the \textbf{tangible lifting map
of} $\vrp$.
\end{defn}

Note  that $s(x)y
= s(xy)$  for $x\in \vrp(R)$, $y\in \tT(\tlU)_0$. \pSkip

\emph{Henceforth we assume  that the \m-valuation $v: R \to M$ is
surjective and that  $\vrp: R \to U$ is a surjective
\m-supervaluation with $e \vrp = v$.} The question arises whether
every ideal $\mfa$ of~$M$ with $\mfa \subset G(\vrp)$ appears as
the ghost value set $G(\psi)$ of some \m-supervaluation $\psi$
covering $v$ with $\vrp \leq \psi \leq \tlvrp$. This is indeed
true.

\begin{construction}\label{const:2.6}
We employ the tangible lifting map $ s: \vrp(R) \ds \to \tlvrp(R)
= \tT(\tlU)_0$ defined above. Assume that $\mfa$ is an ideal of
$M$ contained in $G(\vrp)$. We have $$s(\mfa) \cdot \tlU \ds
\subset s(\mfa) \cup M,$$ since $s(x)y = s(xy) \in s(\mfa)$ for $x
\in \mfa$ and $y \in \tT(U)$. We conclude that $s(\mfa) \cup M$ is
an ideal of $U$. Let
$$ E_\mfa : = E(\tlU, s(\mfa)) = E(\tlU, s(\mfa)\cup M)$$
and $\tlU_\mfa := \tlU/ E_\mfa$. We regard $\tlU_\mfa$  as a
subset of $\tlU$, as indicated in \cite[Convention~3.3.a]{IKR4}.
The map $\pi_{E_\mfa} : \tlU \onto \tlU_\mfa$ is the ideal
compression with ghost kernel $s(\mfa) \cup M$, and
$$ \tlvrp_\mfa  := \tlvrp/ E_\mfa = \pi_{E_\mfa} \circ \tlvrp : R \dss \To \tlU_\mfa$$
is an \m-supervaluation. For any $a\in R$
    $$  \tlvrp_\mfa (a)  = \left\{
\begin{array}{lll}
 \vrp(a) = v(a) & \text{if} \ \vrp(a) \in \mfa , &\\[1mm]
  \tlvrp(a) & \text{else}   \ . &  \\
\end{array}
 \right.
$$
Clearly $\vrp \leq \tlvrp_\mfa \leq \tlvrp$ and $G(\tlvrp_\mfa) =
\mfa$. We call $\tlvrp_\mfa$ \textbf{the tangible lift of $\vrp$
outside $\mfa$}, and we call any such map $\tlvrp_\mfa$ a
\textbf{partial tangible lift of $\vrp$.}
\end{construction}

Let $[\vrp, \tlvrp]$ denote the ``interval'' of the poset $
\Cov_m(v)$ containing all classes $[\psi]$ with $\vrp \leq \psi
\leq \tlvrp$, and let $[0,G(\vrp)]$ be the set of  ideals $\mfa$
of $M$ with $\mfa~ \subset~ G(\vrp)$, ordered by inclusion. By
Lemmas  \ref{lem:2.2} and \ref{lem:2.3} we have a well defined order
preserving map
$$ [\vrp, \tlvrp] \ds \To [0, G(\vrp)],$$ sending each class
$[\psi] \in [\vrp, \tlvrp]$ to the ideal $G(\psi)$. By Theorem
\ref{thm:2.4} this map is injective, and by Construction
\ref{const:2.6} we know that it is also surjective. Thus we we have
proved

\begin{thm}\label{thm:2.7}
The map
$$ [\vrp, \tlvrp] \ds \To [0, G(\vrp)], \qquad [\psi] \mapsto G(\psi),$$
is a well defined order preserving bijection. The inverse of this
map sends an ideal $\mfa \subset G(\vrp)$  to the class
$[\tlvrp_\mfa]$ of the tangible lift of $\vrp$ outside $\mfa$.
\end{thm}

The  poset $\Cov_m(v)$ is a complete lattice (cf.
\cite[Corollary~7.5]{IKR4}). The poset $I(M)$ consisting of the
ideals of $M$ and ordered by inclusion,  is a complete lattice as
well. Indeed, the infimum of a family $(\mfa_i \ds | i\in I)$ in
$I(M)$ is the ideal $\bigcap_i \mfa_i$,   while the supremum is
the ideal $\sum_i \mfa_i = \bigcup_i \mfa_i.$    \{Recall once
more that every subset of $M$ is closed under addition.\} The
intervals $[\vrp, \tlvrp]$ and $[0,G(\vrp)]$ are again complete
lattices, and thus the map $ [\vrp, \tlvrp] \ds \to [0, G(\vrp)]$
in Theorem \ref{thm:2.7} is an anti-isomorphism of complete
lattices. This implies the following

\begin{cor}\label{cor:2.8}
Assume that $(\psi_i \ds | i \in I)$ is a family  of
supervaluations covering $v$ with $\vrp \leq \psi_i \leq \tlvrp$
for each $i \in I$. Let $\bigvee_i \psi_i$ and $\bigwedge_i
\psi_i$ denote respectively representatives of the classes
$\bigvee_i [\psi_i]$ and $\bigwedge_i [\psi_i]$ (as described in
\cite[\S7]{IKR1}). Then
$$G\bigg(\bigvee_i \psi_i \bigg) = \bigcap_i G(\psi_i),  \qquad  G\bigg(\bigwedge_i \psi_i\bigg) = \bigcup_iG(\psi_i).$$
\end{cor} \pSkip

We turn to the case where $\vrp: R \to U$ is a supervaluation,
i.e., the supertropical monoid~ $U$ is a semiring. We want to
characterize the partial tangible lifts $\psi$ of $\vrp$ which are
again supervaluations; in other terms, we want to determine the
subset $[\vrp, \tlvrp] \cap \Cov (v)$ of the interval $[\vrp,
\tlvrp]$ of $\Cov_m(v)$.

The set $Y(v)$ introduced at  the end of \cite[\S7]{IKR4} will
play a decisive role. It consists of the products $a b \in R$ of
elements $a,b \in R$  for which there exists some $a' \in R$ with
$$v(a') < v(a), \quad v(a'b) = v(ab) \neq 0.$$
Henceforth, we call these products $ab$ the
\textbf{$v$-NC-products} (in $R$). Let
$$ \mfq' := \mfq \cup Y(v),$$
where $\mfq$ is the support of $v$, $\mfq = v^{-1}(0)$. As observed in
\cite[\S7]{IKR4}, $\mfq'$ is an ideal of the monoid $(R, \cdot
\;)$, while $\mfq$ is an ideal of the semiring $R$.
\begin{examp}\label{exmp:2.9}

Let $R$ be a supertropical semiring and let $\gm: eR \to M$ be a semiring
homomorphism to a bipotent semiring $M$. Then $$v:= \gm \circ
\nu_R: R \ds \To M$$ is a strict \m-valuation. The $v$-NC-products
are the products $yz$ with $y,z \in U$ such that there exists some
$y' \in R$ with $$\gm(ey') < \gm (ey), \quad \gm(ey'z) = \gm
(eyz).$$ Thus $Y(v)$ is the ideal $D_0(R,\gm)$ of the
supertropical semiring $R$, introduced in \cite[Definition~4.8]{IKR4}.

\end{examp}
\begin{prop}\label{prop:2.10} If $\vrp$ is a supervaluation then
$\vrp(\mfq')$ is contained in the ghost value set~ $G(\vrp)$.
\end{prop}

\begin{proof} We have seen in \cite[\S7]{IKR4} that $\vrp(Y(v)) \subset M$. Since $\vrp(\mfq) = \{
0\}$, this implies that $\vrp(\mfq') \subset M \cap \vrp(R) =
G(\vrp)$.
\end{proof}

\begin{rem}\label{rem:2.11} Here is a more direct argument that $\vrp(Y(v))
\subset M$, than given in the proof of
\cite[Theorem~7.12.i]{IKR4}. If $x \in Y(v)$, then we have $a',a,b
\in R$ with $x =ab$, $v(a') < v(a)$, $v(a'b) = v(ab) \neq 0$.
Clearly $\vrp(x) = \vrp (a) \vrp (b)$ is an NC-product in the
supertropical semiring~ $U$  \cite[Definition~4.2]{IKR4},
and thus $\vrp(x)$ is ghost, as  already observed in
\cite[Theorem~1.2]{IKR4}.
\end{rem}

\begin{lem}\label{lem:2.12}
Assume that $\vrp: R \to U$ is a surjective tangible
\m-supervaluation covering ~ $v$. Then $\vrp(R \sm \mfq) = \tT(U)$,
$v(R) = M$, and $\vrp(Y(v)) = S(U)$.\footnote{Recall that $S(U)$
denotes the set of tangible NC-products in $U$
\cite[Definition~4.2]{IKR4}.}
\end{lem}

\begin{proof} a) We have $U = \vrp(R) \cup v(R)$, $\vrp(R \sm \mfq) \subset
\tT(U)$, and $v(R) \subset M$. Since $U = \tT(U) \dot \cup M$,
this forces $\vrp(R \sm \mfq) = \tT(U)$ and $v(R) = M.$ \pSkip
b) Let $c \in Y(v)$. There exist $a,b,a' \in R$ with $c =ab$,
$v(a') < v(a)$, $v(a'b) = v(ab) \neq 0.$ It follows that $\vrp(c)
= xy \neq 0$ with $x:= \vrp(a)$, $y:= \vrp(b)$, $v(a') < ex$,
$v(a') y = exy.$ Thus $\vrp(c)$ is an NC-product in $U$. Moreover
$\vrp(c)$ is tangible, hence $\vrp(c) \in S(U).$  Thus $\vrp(Y(v))
\subset S(U).$ \pSkip
c) Let $x \in S(U)$  be given. Then $x = yz \in \tT(U)$ with $y,z
\in U$ and $y' < ey$, $y'z = eyz \neq 0$ for some $y' \in M$.
Clearly $y,z\in \tT(U)$. We choose $a, b, a' \in R$ with $\vrp(a)
= y$, $\vrp(b) = z$, $v(a') = y'.$ Then $ey = v(a)$, $ez = v(b)$,
and it follows that $v(a') < v(a)$, $v(a'b) = v(ab) \neq 0$. Thus
$ab \in Y(v)$ and $x = \vrp(ab)$. This proves that $S(U) \subset
\vrp (Y(v))$.
\end{proof}

\begin{thm}\label{thm:2.13} Assume that $\vrp: R \to U$ is a
supervaluation, i.e., $U$ is a semiring. Let $\chvrp$ denote the
tangible lift of $\vrp$ outside the ideal $v(\mfq') = \{ 0 \} \cup
v(Y(v))$ of $M$,
$$ \chvrp := (\tlvrp)_{v(\mfq')} : R \ds \To \chU := \tlU/E_{v(\mfq')}$$
(cf. Construction \ref{const:2.6}).
\begin{enumerate} \eroman
    \item $\chvrp$ is again a supervaluation. More precisely,
    $\chvrp$ coincides with the supervaluation $(\tlvrp)^\wedge$
    associated to the tangible lift $\tlvrp: R \to \tlU$ of $\vrp$
    (cf. \cite[Definition~7.7]{IKR4}).

    \item If $\psi$ is an \m-supervaluation covering $v$ with $\vrp \leq \psi \leq
    \tlvrp$, then $\psi$ is a supervaluation iff $\psi \leq
    \chvrp$. Thus
    $$ [\vrp, \tlvrp] \cap \Cov(v) = [\vrp, \chvrp].$$
\end{enumerate}
\end{thm}

\begin{proof} (i): $(\tlvrp)^\wedge$ is the map $\tlvrp/ E(\tlU, S(\tlU)) = \pi_{E(\tlU, S(\tlU))} \circ
\tlvrp$ from $R$ to $\chU:= \tlU/ E(\tlU, S(\tlU)).$ Applying Lemma
\ref{lem:2.12} to $\vrp$, we have
$$ S(\tlU) \cup \{ 0 \} = \tlvrp (\mfq') = s\vrp(\mfq')$$
where $s: \vrp(R) \to \tT(\tlU)_0$ is the tangible lifting
map for $\vrp$. Moreover $\vrp(\mfq') = v(\mfq')$ by Proposition
\ref{prop:2.10}. Thus $ \chU = \tlU/ E_{v(\mfq')}$ and
$(\tlvrp)^\wedge = \tlvrp/ E_{v(\mfq')} = \chvrp.$ \pSkip
(ii): If $\psi$ is a supervaluation, then we know by Proposition
\ref{prop:2.10} that $G(\psi) \supset v(\mfq') = G(\chvrp)$, and
hence by Theorem \ref{thm:2.4} that $\psi \leq \chvrp$. Conversely,
if $\psi \leq \chvrp$, then $\psi$ is a supervaluation, since
$\chvrp$ is a supervaluation (cf. \cite[Proposition~7.6]{IKR4}).
\end{proof}
\begin{defn}\label{defn:2.14}$ $ \begin{enumerate} \eroman
    \item Given a supervaluation $\vrp:= R \to U$ covering $v$ we
    call
    $$ \chvrp := (\tlvrp)^\wedge: R \ds \To \chU = (\tlU)^\wedge$$
the \textbf{almost tangible lift of $\vrp$} (to a supervaluation)
and we call $[\chvrp] \in \Cov(v)$ the almost tangible  lift (in
$\Cov(v)$) of the class $[\vrp] \in \Cov(v)$. \pSkip

    \item If $\chvrp = \vrp$, we say that $\vrp$  itself  is \textbf{almost tangible}.
\end{enumerate}

\end{defn}

\begin{rems}\label{rem:2.15}
$ $ \begin{enumerate} \ealph
    \item Clearly $\vrp$ is almost tangible iff $G(\vrp) =
    v(\mfq')$. A subtle point here is that there may
    nevertheless exist elements $a \in R \sm \mfq'$  with
    $\vrp(a)$ ghost.

    \item If $\vrp$ is any supervaluation, then
    $\chvrp$ is almost tangible.

    \item If $v$ happens to be a valuation, i.e., $M$ is
    cancellative, then $\chvrp = \tlvrp$.
\end{enumerate}
\end{rems}

\begin{prop}\label{prop:2.16}
If $\psi$ is an almost tangible supervaluation dominating the
supervaluation~$\vrp$ (but not necessarily covering $v$), then
$\psi$ dominates $\chvrp$.
\end{prop}

\begin{proof} $\tlpsi \geq \tlvrp$, and hence $\psi= (\tlpsi)^\wedge \geq (\tlvrp)^\wedge = \chvrp.$
\end{proof}

\section{Tangible and mixing transmissions}\label{sec:3}

The  intent  of this section is to display any transmission
$\al: U \to V$ of  supertropical monoids as a product $\al =
\al_m \circ \al_t$ of a \emph{tangible transmission} $\al_t$ and a
\emph{``mixing'' transmission} $\al_m$, and to study properties of
these  factors. Tangible transmissions were
introduced in \cite[Definition 2.1]{IKR4}, while mixing
transmissions are introduced below. A key result for this
factorization is provided by Theorem
\ref{thm:3.2} on TE-relations \cite[Definition~1.7]{IKR4}.

\begin{defn} \label{defn:3.1} A   TE-relation $E$ on a
supertropical monoid $U$ is called \textbf{ghost separating}, if no element
$x\in \tT(U)$ with $x \not \sim_E 0 $ is $E$-equivalent to an
element $y \in \tG(U)$. In other terms, $E$~is ghost separating
iff $\pi_E: U \to U/E$ is a tangible transmission \cite[Definition~2.3]{IKR4}.
\end{defn}

\begin{thm} \label{thm:3.2} Let $E$ be a TE-relation on a
supertropical monoid $U$. Define the binary relation $\tlE$
on $U$  by
%
 \begin{equation}\renewcommand{\theequation}{$*$}\addtocounter{equation}{-1}\label{eq:str.1}
 x \sim_{\tlE}y  \dss \Leftrightarrow
\left \{
\begin{array}{lll}
x \sim _E y , & \text{and for any } z \in U &\\[1mm]
  \text{either} &    xz, yz \in \tT(U) \cup \00, \\[1mm]
  \text{or} &    xz, yz \in eU, \\[1mm]
  \text{or} &   xz \sim_E 0 \ ( \text{hence } yz \sim_E 0). \\
\end{array}%
\right.
\end{equation}
\begin{enumerate} \eroman
    \item $\tlE$ is a TE-relation on $U$ with $\tlE \subset E$ and
    $\tlE | M = E|M$.
    \item $\tlE$ is ghost separating.

    \item If $F$ is a ghost separating TE-relation on $U$ with $F \subset
    E$, then $F \subset \tlE$.
\end{enumerate}

\end{thm}

\begin{proof} (i): It is obvious that $\tlE$ is a multiplicative
equivalence relation on $U$ and $\tlE | M = E|M$. Assume that $x
\in U$ and $ex \sim_\tlE 0$. Then $ex \sim_E 0$, hence $x \sim_E
0$, implying that  $x \sim_\tlE 0$. Thus $\tlE$ is a
TE-relation. Clearly $\tlE \subset E$. \pSkip
(ii): Assume that $x\in\tT(U)$, $y \in \tG(U)$, but $x \not \sim_E
0$ and $y \not \sim_E 0$.  Then, $x$ and $y$ are not
$\tlE$-equivalent, since the second condition in $(*)$ fails for $z =1$. Thus $\tlE$ is ghost
separating. \pSkip
(iii): Assume that $x,y \in U$ and $x \sim_F y$. Then $x \sim_E
y$. Given $z \in U$, since $F$ is ghost separating,
either $xz, yz \in \tT(U)$ or $xz,yz \in \tG(U)$, or $xz \in [0]_F
\subset [0]_E$ \footnote{Recall that $[0]_F$ denotes the
$F$-equivalence class of $0 \in U$.}. Thus $x \sim_\tlE y$. This
proves that $F \subset \tlE$.
\end{proof}

\begin{defn} \label{defn:3.3} $ $
\begin{enumerate} \ealph
    \item
We call $\tlE$ the \textbf{ghost separating refinement} of the
TE-relation $E$.

\item If $E$ is an MFCE-relation, we alternatively say that $\tlE$
is the \textbf{tangible refinement} of $E$, to be compatible with the definition of a tangible MFCE-relation in
\cite[Definition 2.4]{IKR4}.

\end{enumerate}
\end{defn}

In the second condition on the right  side of $(*)$, defining
 $\tlE$,  we may discard the elements~ $z$ of $eU$, since then we
have $xz,yz \in eU$ for free. We may also use the sets $\tT(U)$,
$\tG(U)$, $[0]_E$ instead of $\tT(U) \cup \00$, $eU$, $[0]_E$,
although $\tT(U)$ or $\tG(U)$ may not be closed under
multiplication. This leads to the following description of
$\tlE$, which perhaps looks more natural than $(*)$.

\begin{rem} \label{rem:3.4} Let $x,y \in U$. The following are
equivalent.
\begin{enumerate} \eroman
    \item $x \sim_{\tlE} y$.
    \item  $x \sim_{E} y$, and for every $z \in \tT(U)$ the
    elements $xz,yz$ are both contained in one of the sets $\tT(U),
    \tG(U), [0]_E$.
\end{enumerate}

\end{rem}

The description of $\tlE$ becomes much simpler if $U$ is
unfolded, i.e., $\tT(U)_0$ is closed under multiplication.
\begin{rem} \label{rem:3.5}
Assume that $E$ is a TE-relation on an unfolded supertropical
monoid $U$. Let $x,y \in U$. Then $x \sim_{\tlE} y$ iff $x \sim_E
y$ and the elements $x,y$ are both contained in one of the sets
$\tT(U)$, $\tG(U)$, $[0]_E$. Indeed, now the last property is
inherited by the pair $xz,yz$ from the pair $x,y$ for any $z \in
U$.
\end{rem}

We are ready for introducing the class of ``mixing''
transmissions. It will be convenient to use the following
catch-phrases: We say that two elements $x,y$ of $U \sm \{ 0\}$
are \textbf{of different kind}, if either $x\in \tT(U)$, $y \in
\tG(U)$, or $x\in \tG(U)$, $y \in \tT(U)$. Otherwise we say that
$x,y$ are \textbf{of the same kind}.

\begin{defn} \label{defn:3.6} Let $\al: U \to V$ be a transmission
between supertropical monoids. We say that $\al$ is
\textbf{mixing}, if $\ial(0) = \{ 0 \}$, and if for any two
different elements $x,y$ of $U \sm \00$ with $\al(x) = \al(y)$
there exists some $z\in U$ with $\al(xz) \neq 0$ (hence also
$\al(yz) \neq 0$) and $xz$, $yz$  of different kind.
\end{defn}

We start the study of this new class of transmissions with some
simple observations.
\begin{rems} \label{rem:3.7} $ $
\begin{enumerate} \eroman
    \item If $\al : U \to V$ is a transmission with $\ial(0) =
    \00$, and $E:= E(\al)$ is the equivalence relation determined  by $\al$
    ($x \sim_E y$ iff $\al(x) = \al(y)$), then $\al$ is mixing iff the
    ghost separating refinement $\tlE$ of $E$ is trivial, i.e., $x \sim_\tlE
    y$ implies $x = y$. This is obvious from the definition of
    $\tlE$ (cf. Remark \ref{rem:3.4}).

    \item If $\al : U \to V$ is mixing, then the ghost part $\al^\nu: eU \to
    eV$ of $\al $ is injective. Indeed, if $x,y \in eU$  and $\al(x) =
    \al(y)$, then $xz,yz \in eV$ for every $z \in U$. This forces
    $x=y$. Thus every surjective mixing transmission is a fiber contraction.

    \item Let  $\al_1 : U \to U_1$ and $\al_2 : U_1 \to U_2$ be
    transmissions, and assume that $\al_2$  is injective. Then  $\al_2 \circ
    \al_1$ is mixing iff $\al_1$ is mixing. Every mixing
    transmission $\al : U \to V$ is the composite $ i \circ \al_0$
    of a surjective mixing transmission $\al_0: U \to \al (U)$ and
    an inclusion map $i: \al(U) \into V$.

    \item  If $\al : U \to V$ is mixing and $j: U' \to U$ is an
    injective transmission,  then $\al \circ j: U' \to V$ is again
    mixing.
\end{enumerate}
\end{rems}

\begin{lem} \label{lem:3.8} Assume that $\al: U \to V$ and $\bt: V \to
W$ are transmissions and that $\bt \al : U \to W$ is mixing. Then
$\al$ is mixing.
\end{lem}

\begin{proof} Since $(\bt \al) ^{-1} (0) = \00$, also $\ial(0) =
\00$. Let $x,y$ be different elements in $U \sm \00$ with $\al(x)
= \al(y)$. Then $\bt \al(x) = \bt \al(y)$. Since $\bt\al$ is
mixing,  there exists  some $z \in U$ with $\bt \al(xz) \neq 0$,
 hence $\al(xz) \neq 0$, and $xz$, $yz$  of different kind. This
proves that $\al$ is  mixing.
\end{proof}

\begin{lem} \label{lem:3.9}
Assume that $\al: U \to V$ is a surjective transmission which is
both tangible and mixing, Then $\al$ is an isomorphism.
\end{lem}

\begin{proof} A bijective transmission is  an isomorphism. Thus
it suffices to verify that $\al$ is injective. Let $x,y$ be
different elements of $U$. Suppose that  $\al(x) = \al(y)$. Since
$\ial(0) = \00 $, both  $x,y$ are not zero. Since $\al$ is mixing,
we conclude that there exists some $z \in U$ with $\al (xz) \neq
0$, $\al (yz) \neq 0$, and, say $xz \in \tT(U)$, $yz \in \tG(U)$.
Since $\al$ is tangible, it follows that $\al(xz) \in \tT(V) \cup
\00$, $\al(yz) \in eV$. But  $\al(xz) = \al(yz)$, which forces
$\al(xz) =0 $, a contradiction. This proves that $\al$
is injective.
\end{proof}

We are ready for the factorization theorem, announced
at  the beginning of the section.
\begin{thm} \label{thm:3.10} Let $\al : U' \onto U$ be a surjective
transmission, $M' := eU'$, $M := eU$, and  $$\gm := \al^\nu: M'
\Onto M.$$

\begin{enumerate} \eroman
    \item
There exists a commuting triangle
\begin{equation}\label{eq:3.0}
\xymatrix{    &   & V \ar @{>}[rrd]^{\al_m} & &
 \\
U'  \ar @{>}[rrrr]_{\al} \ar @{>>}[rru]^{\al_t} & & & & U  }
\end{equation}
where $\al_t$ is  a surjective tangible transmission covering $\gm$ (in
particular $eV = M$) and $\al_m$ is a mixing transmission over $M$.

    \item
Given a factorization $\al = \al_m \al_t$ as described in (i), the
following holds. If
    $$
\xymatrix{ \al: U'  \ar @{>>}[r]^{\bt}   & W  \ar @{>}[r]^{\mu}&
U   }
$$
is a factorization of $\al$ with $\bt$ surjective and tangible and $\mu$ a fiber
contraction over $M$ (in particular $eW = M$), then there exists a
unique fiber contraction $\zt: W \to V$ over $M$ such that $\zt
\bt = \al_t$, which forces $\al_m \zt = \mu$. If in addition
$\mu$ is mixing, then~ $\zt$ is an isomorphism.  \end{enumerate} We
indicate this by the following diagram:
 $$
\xymatrix{   &   & W \ar @{>}[rrdd]^{\mu} \ar @{.>}[d]^{\zt}& &
\\ &   & V \ar @{>}[rrd]_{\al_m} & &
 \\
U'  \ar @{>}[rrrr]_{\al} \ar @{>>}[rru]_{\al_t} \ar
@{>>}[rruu]^{\bt} & & & & U . }
$$
\end{thm}

\begin{proof} a) Let $E := E(\al)$, and take
$$ \al_t := \pi_\tlE: U' \ds \To U' /\tlE = : V .$$
The transmission $\al_t$ is tangible and covers $\gm$ (cf. Theorem
\ref{thm:3.2}). Since $\tlE \subset E$, there is a unique map
 $\al_m: V \to U$ with $\al = \al_m \al_t$. Since $\al$ and
 $\al_t$ are surjective transmissions, also~$\al_m$ is a
 surjective transmission. Since both $\al$ and $\al_t$ cover
 $\gm$, we have $\al^\nu_m = \id_M$. Thus~ $\al_m$ is a fiber
 contraction over $M$.

 We verify that $\al_m$ is mixing. Let $x,y \in U'$ be given. Let
 $\brx := [x]_\tlE$, $\bry := [y]_\tlE$ and assume that $\brx \neq
 \bry$, i.e., $x \not \sim_\tlE
 y$. Assume also  that $\al_m(\brx) =  \al_m(\bry)$, i.e.,
 $x \sim_E  y$, in other terms, $\al(x) = \al(y)$.
 By Remark \ref{rem:3.4} we conclude that there exists some $z\in
 U$ with $\al(xz) \neq 0$, $\al(yz) \neq 0$, and $xz, yz$  of
 different kind. Then,  with $\brz := [z]_\tlE$, we have $\al_m(\brx \brz) \neq
 0$, $\al_m(\bry \brz) \neq
 0$, and   $\brx \brz, \bry \brz$ of
 different kind,  since $\tlE$ is ghost separating. This proves that
 $ \al_m$ is mixing.
\pSkip
b) We continue with the factorization $\al = \al _m \al_t$ as just
constructed. Let $\al = \mu \bt$ be a factorization of $\al$ with
$\bt: U' \to W$ tangible and surjective, and $\mu: W \to U$ a
fiber contraction over~ $M$. Since $\mu ^\nu = \id_M$, this forces
$\bt^\nu = \gm$. We have $E(\bt) \subset E(\al) = E$, and $E(\bt)$
is ghost separating. Thus $E(\bt) \subset  \tlE$. We conclude that
there exists a surjective transmission $\zt: W \to U' / \tlE = V$
with $\zt \bt = \al_t$. Since $\bt^\nu = \al_t^\nu = \gm$, it
follows that $\zt^\nu = \id_M$, i.e., $\zt$ is a fiber contraction
over $M$, and thus  $\al_m \al_t = \al_m \zt \bt = \mu \bt$. Since
$\bt$ is surjective, this implies $\mu = \al_m \zt$.
\pSkip
c) Assume in addition that $\mu$ is mixing. Since  $\al_t = \zt
\bt$ and $\bt$ are tangible, also $\zt$ is tangible. Since  $ \mu
= \al_m \zt $ is mixing, if follows by Lemma \ref{lem:3.8} that
$\zt$ is mixing, and then by Lemma \ref{lem:3.9} that $\zt$ is an
isomorphism.
 \pSkip
 d)
It is now evident that the
 statement (ii) of the theorem holds for any
 factorization $\al = \al_m \al_t$ as described in part (i), and not
 only for the one constructed in step a) of the  proof.
\end{proof}

\begin{rem}\label{rem:3.11}
Assume that in Diagram \eqref{eq:3.0} the monoid $U'$ is a (supertropical) semiring. Then also $V$ is a semiring, cf. \cite[Theorem 1.6.(ii)]{IKR4}, and consequently $U$ is a semiring. Thus the whole Diagram \eqref{eq:3.0} is situated in the category of supertropical semirings. In particular the mixing property of $\al_m$ takes place in this category.

\end{rem}
Theorem \ref{thm:3.10} can be readily generalized to a
factorization theorem of a transmission which is not necessarily
surjective. 
\begin{cor} \label{cor:3.11}
Assume that $\al: U' \to U$ is any transmission. Let $\gm : =
\al^\nu : M' \to M$ denote the ghost part of $\al$.
\begin{enumerate} \eroman
    \item There exists a surjective tangible transmission $\al_t: U' \to
    V$ and a mixing transmission $\al_m: V \to
    U$  covering the injection $i : \gm(M) \into M'$ such
    that $\al = \al_m \circ \al_t$.

    \item Given a second factorization $\al = \mu \circ \bt$ with $\bt: U' \to
    W$  surjective and tangible, and
    $\mu: W \to U$ mixing, there exists  a unique isomorphism  $\zt: W \isoto V$ such that $\zt \bt = \al_t$ and  $\al_m \zt  = \mu.$
\end{enumerate}
\end{cor}

\begin{proof} (i): Let $\al_0: U' \to \al(U')$ denote the transmission
obtained from $\al$ by restricting  the target $U$ to $\al(U')$.
We have a factorization $\al_0 = (\al_0)_m \circ (\al_0)_t$ with
the properties stated in Theorem \ref{thm:3.10}.
Let $\al_t :=
(\al_0)_t$ and $\al_m := j \circ (\al_0)_m$ with $j$ the inclusion
map from $\al(U')$ to $U$. Then $\al = \al_m \circ \al_t$. Here $\al_t$ is surjective and tangible, while $\al_m$ is mixing, as follows form Remak \ref{rem:3.7}.(iv).
\pSkip
(ii): By Theorem \ref{thm:3.10}.(ii) there is a unique isomorphism $\zt: W \to \al_0(U)$ with $\zt \bt = (\al_0)_t = \al_t$. Since $\al = \al_m \al_t = \al_m \zt \bt = \mu \bt $, and $\bt$ is surjective,
this forces $\al_m \zt = \mu$.
\end{proof}

\begin{defn} \label{defn:3.12} We call $\al_t$ ``the'' \textbf{surjective tangible
part} and $\al_m$ the \textbf{mixing part} of the transmission
$\al$. $\al  = \al_m \circ \al_t$ is called the
\textbf{$(\t,\m)$-factorization} of $\al.$ When  $\al$ is surjective, we
call $\al_t$ briefly the \textbf{tangible part} of $\al$.
\end{defn}

\begin{remark}\label{rmk:3.13} Let $\al : U \to V$ be any
transmission with ghost part $\gm : = \al^{\nu}: M \to N$. By
\cite[Theorem 1.13]{IKR4} we have a unique factorization
$$ \al = \bt \circ \pi_{F(U,\gm)}, $$
where  $\bt$ is a fiber contraction over $N = eV$. Since
$\pi_{F(U,\gm)}$ is tangible and surjective, it is obvious that
$$ \al_t = \bt_t \circ \pi_{F(U,\gm)}, \qquad \al_m = \bt_m.$$
\end{remark}

The remark indicates that finding $(\t,\m)$-factorizations is
essentially an issue on fiber contractions, and thus can be turned
into a problem about MFCE-relations.

From part a) of the proof of Theorem \ref{thm:3.10} now the
following is clear.

\begin{schol}\label{schol:3.14}
Let $E$ be an MFCE-relation on a supertropical monoid $U$. Then
the transmission $\pi_E: U \to U/ E$ has the tangible part
$(\pi_E)_t = \pi_{\tlE}$ and the mixing part $(\pi_E)_m = \pi_{E /
\tlE}$. Here $E / \tlE$ denotes the equivalence relation induced
by $E$ on $U / \tlE$, i.e., for~ $x,y \in U$
$$[x]_\tlE \sim _{E / \tlE} [y]_\tlE \dss \Leftrightarrow x \sim_E y.  $$

\end{schol}

\begin{examp}[The relation $\tlE(\nu)$]\label{exmp:3.15} Let $U$ be
a supertropical monoid and  $M :=  eU$. For any subset $L$ of $U$
and $x \in U$ we define
$$[L:x] := \{ z\in U \ds | zx \in L \}.  $$
The equivalence relation $E(\nu) := E(\nu_U)$ induced by the ghost
map $\nu_U$  clearly is the coarsest MFCE-relation on $U$. Notice
that $x \sim _{E(\nu)} y$ iff $ex = ey$. Theorem \ref{thm:3.2}
tells us that the  relation $\tlE(\nu) := (E(\nu))^\sim$  is the
coarsest  ghost separating MFCE-relation  on $U$.
Furthermore,  by that theorem two elements $x,y$ of $U$ are
$\tlE(\nu)$-equivalent iff $ex =ey$ and for every $z\in U$ both
$xz$ and $yz$ lie in $\tT(U)_0$ or in $\tG(U)_0$. This can also be
phrased as follows.
$$\begin{array}{rcll}
  x \sim_{\tlE(\nu)} y & \Leftrightarrow & ex = ey, & [M:x ] =
  [M:y]\\[1mm]
   & \Leftrightarrow & ex = ey, & [\tG(U):x ] = [\tG(U):y]\\[1mm]
   & \Leftrightarrow & ex = ey, & [\tT(U):x ] = [\tT(U):y].\\
\end{array}$$
\end{examp}

\begin{remark}\label{rem:3.16}
If $E$ is any  MFCE-relation,  we read off from the preceding
remark and Theorem \ref{thm:3.2} that $\tlE = E \cap \tlE(\nu)$.
Since any equivalence relation finer than $\tlE(\nu)$ is ghost
separating, it is  clear   form Remark \ref{rmk:3.13}
that $E \cap \tlE(\nu)$ is the coarsest ghost separating
MFCE-relation finer than $E$.
\end{remark}

\begin{examp}[The relation $E_t$]\label{exmp:3.17} We return to the
fiber conserving  equivalence relation  $E_t := E_{t,U}$ mentioned
near   \cite[Definition 2.3]{IKR4}. For $x,y \in U$
$$\begin{array}{rcll}
  x \sim_{E_t} y & \Leftrightarrow & ex = ey, &  \text{and either } x= y \text{ or  } x,y \in \tT(U).
\end{array}$$
Clearly $\tlE(\nu) \subset E_t$. If $E_t$ is multiplicative,  then
$E_t$ is MFCE and ghost separating, hence $E_t \subset \tlE(\nu).$
We conclude that $\tlE(\nu) = E_t$ iff for any elements $x,y,z \in
\tT(U)$ with $ex= ey$ the elements $xz, yz$ are both tangible or
equal. Certainly this happens if $\tT(U) \cup \00$  is closed
under multiplication, but there are other cases.
\end{examp}

As in \cite[\S4]{IKR4} (and for $U$ a semiring ordering in
\cite[\S7]{IKR1}) we denote the poset consisting of all
MFCE-relations on $U$ by $\MFC(U).$ Recall that $\MFC(U)$ is a
complete lattice with top element $E(\nu_U)$ and bottom element
$\diag (U)$, cf. \cite[\S7]{IKR1}. We take a  brief look at the
sets of those $E \in \MFC(U)$ for which the associated fiber
contraction $\pi_E: U \to U/E$ is either tangible or mixing.

\begin{defn}\label{def:3.18} $ $
\begin{enumerate} \ealph
    \item Let $E \in \MFC(U)$. We say that $E$ is
    \textbf{tangible} (resp. \textbf{mixing}) if the transmission~
    $\pi_E$ is tangible (resp. mixing).

    \item We denote the set of all tangible $E \in \MFC(U)$ by
    $\MFCt(U)$,  and the set of all mixing $E \in \MFC(U)$  by
    $\Mix(U).$ We regard both $\MFCt(U)$ and $\Mix(U)$ as
    subposets of the lattice $\MFC(U).$
\end{enumerate}
\end{defn}
N.B. An MFCE-relation $E$ on $U$ is tangible iff the TE-relation
$E$ is ghost separating. Tangible MFCE-relations had already been
introduced in \cite[Definition 2.4]{IKR4}.

\begin{schol}\label{schol:3.19} It follows from Remark
\ref{rem:3.16}, that $\MFCt(U)$ is the set of all $E \in \MFC(U)$
with $E \subset \tlE(\nu)$, while $\Mix(U)$ is the set of all $E
\subset \MFC(U)$ with $ E \cap \tlE(\nu) = \{ \diag U \} $. Thus
both $\MFC_t(U)$ and $\Mix(U)$ are lower subsets of $\MFC(U)$, and
$$ \MFCt(U) \cap \Mix(U) = \{ \diag U \}.$$
Moreover,  $\MFCt(U)$ is a complete sublattice  of $\MFC(U).$
\end{schol}

\begin{lem}\label{lem:3.20}
Assume that $(E_i \ds | i \in I)$ is a family of MFCE-relations on
$U$ with the property that for any two indices $i,j \in I$ there
exists some $k \in I$ with $E_i \subset E_k $ and $E_j \subset E_k
$. Let $E:= \bigcup_{i \in I} E_i \subset U \times U. $
\begin{enumerate} \eroman
    \item $E$ is again an MFCE-relation on $U$. Thus
    $$ \bigcup_{i \in I} E_i =  \bigvee_{i \in I} E_i \ .$$

    \item For any $L \in \MFC(U)$ we have
    $$  \bigg( \bigvee_{i \in I} E_i \bigg) \wedge L = \bigvee_{i \in I} (E_i \wedge L ) =
    \bigcup_{i \in I} (E_i \cap  L) \ .  $$

    \item In particular $(L  = \tlE(\nu))$
    $$ \bigg(\bigvee_{i \in I} E_i \bigg)^{\sim} = \bigvee_{i \in I} \tlE_i = \bigcup_{i \in I} \tlE_i \ .$$
\end{enumerate}
\end{lem}

\begin{proof} (i): An easy check  starting from the definition of
an MFCE-relation. (First verify that~ $E$ is an equivalence
relation.) \pSkip
(ii): Apply (i) to the family ($E_i \cap L \ds | i \in I$). \pSkip
(iii): Now clear, since for any $F \in \MFC(U)$ we have $\tlF = F
\cap \tlE(\nu)$ (cf. Remark \ref{rem:3.16}).
\end{proof}

\begin{thm}\label{thm:3.21} For every $E \in \Mix(U)$ there exists
a maximal element $F$ of $\Mix(U)$ with~ $E \subset F$.
\end{thm}

\begin{proof} Let $(E'_i \ds | i \in I)$ be a
chain in the poset $\Mix(U)$ with $E \subset E'_i$ for all $i \in
I$. Then $E' := \bigcup_i E'_i$ is a mixing MFCE-relation on $U$
by Lemma \ref{lem:3.20}.(iii).  The claim of the theorem follows by
Zorn's lemma.
\end{proof}

 $\Mix(U)$ will be a sublattice of $\MFC(U)$ only in  rather
 degenerate cases.

\section{The $(t,m)$-factorization of the composite of two transmissions}\label{sec:4}

Given two transmissions  $\al:U \to V$ and $\bt:V \to W $ with $(t,m)$-factorizations $\al = \al_m \circ \al_t$,
$\bt = \bt_m \circ \bt_t$, we intend to determine ``the'' $(t,m)$-factorization of the composite $\bt \circ \al : U \to W$.

First, we have the obvious fact that the composite of two tangible transmissions is again tangible. This is also the case for mixing transmissions.

\begin{prop}\label{prop:4.1} If $\al:U \to V$ and $\bt:V \to W $  are mixing transmissions, then  $\bt \circ \al : U \to W$ is mixing as well.
\end{prop}
\begin{proof} 1) $(\bt \al )^{-1} = \{ 0_U\} $, since $\bt^{-1}(0_W) = \{ 0_V\}$ and
$\al^{-1}(O_V) = \{ 0_U\}$.
\pSkip
2) We now verify the claim in the case that $\al$ is surjective. Given $x,y \in U \sm \00$, where $x \neq y$ and $\bt \al (x) = \bt \al (y)$, we need to find some $z \in U$, necessarily $\neq 0$,  such that $\bt \al (xz) \neq 0 $ and $xz$ and $yx$ are of different kind.

Assume first that $\al(x) = \al(y)$. Since $\al$ is mixing, there exists $z \in U$ for which  $xz \neq 0$, $yz \neq 0$ and $xz$, $yz$ are of different kind. There remain the case that $\al(x) \neq \al(y)$. By Step~ 1 it is clear that both $\al(x)$ and  $\al(y)$ are not zero (and also $\bt\al(x) = \bt \al(y) \neq 0$). Since~ $\bt$ is mixing, there is some $v \in V \sm \00 $ with $\al(x)v$ and $\al(y)v$ of different kind. As $\al$ is surjective, there is some $z \in U \sm \00$ with $\al(z) =v$. Then  $\al(x)v = \al(xz)$,  $\al(y)v = \al(yz)$. Since $\al(x)v$ and $\al(y)v$ are of different kind, clearly also $xz$ and $yz$ are of different kind.
\pSkip
3) When $\al$ is not surjective, we have a factorization $\al = j \al_0 $ where $\al_0: U \onto  \al(U)$ is a  surjective transmission and $j : \al (U) \into V$ is an injective transmission. Then $\al = (\bt j ) \circ \al_0$. We conclude by Lemma \ref{lem:3.8} that $\al_0$ is mixing, since $j \al_0$ is mixing, and by Remark \ref{rem:3.4}.(iv) that $\bt j$ is mixing. Finally, by Step 2 we obtain that $\bt \al = (\bt j) \circ \al_0$ is mixing.
\end{proof}

\begin{thm}\label{thm:4.2}
Assume that
  $\al:U \to U'$ and $\bt:U' \to U'' $ are transmissions with
   $(t,m)$-factorizations $\al = \al_m \circ \al_t$, $\bt = \bt_m \circ \bt_t$, having the surjective tangible parts
   $\al_t: U \to~ V'$, $\bt_t:U' \onto V''$ and the  mixing parts $\al_m: V' \to U'$, $\bt_m:V'' \onto U''$, respectively (cf. Definition ~\ref{defn:3.12}). Consider also the  $(t,m)$-factorization of $\rho:= \bt_t \circ \al_m:V' \to V''$ with surjective tangible part $\rho_t:V' \to W$ and mixing part $\rho_m: W \to V''$, so that all together the following commutative diagram of transmissions appears:
   $$
\xymatrix{
& & && W \ar @{>}[d]^{\rho_m}
\\
  & & V' \ar @{>}[d]^{\al_m} \ar @{>}[rr]^{\rho} \ar @{>}[rru]^{\rho_t} & & V'' \ar @{>}[d]^{\bt_m}
  \\
U  \ar @{>}[rr]_{\al} \ar @{>}[rru]^{\al_t}   &  & U' \ar @{>}[rr]_{\bt} \ar @{>}[rru]^{\bt_t} &  & U'' }
$$
Then the transmission $\bt \al : U \to U'' $ has the surjective tangible part $\rho_t \al_t$ and the mixing part $\bt_m \rho_m$.
\end{thm}

\begin{proof} We have the factorization
$$ \bt \al = (\bt_m \rho_m )(\rho_t \al_t).$$
  The transmission $\rho_t \al_t$ is surjective tangible, while  $\bt_m \rho_m$ is mixing by Proposition \ref{prop:4.1}. The assertion follows from the uniqueness of $(t,m)$-factorizations, stated in Corollary~ \ref{cor:3.11}.(ii).
\end{proof}

\section{Ideal-generating sections and tyrants}\label{sec:5}

We exhibit a class of  special elements of the poset
$\Mix(U)$, introduced in \S\ref{sec:3}.

\begin{defn}\label{defn:5.1} Let $U$ be a supertropical monoid with $M :=e
U.$ A map $\s: M \to U$ is  called an \textbf{ig-section}  (=
\textbf{ideal-generating section}) in $U$, if
\begin{alignat*}{3}
&(\SC1)\quad && \forall x \in M : \quad e \s(x) =x,\\
& & & \quad \text{i.e., $\s$ is a section of  $\nu_U : U \onto M;$}\\
&(\SC2)\quad && \s(M) U \subset M \cup \s(M),\\
& & & \quad \text{i.e., $\mfA(s) := M \cup \s(M)$ is an ideal of  $U$
containing $M$.}
\end{alignat*}
\end{defn}

\begin{remark}\label{rem:5.2}
Assuming $(\SC1)$,  the condition $(\SC2)$ is equivalent to:
\begin{alignat*}{3}
&(\SC2')\quad && \text{If } x\in M, y \in U, s(x)y \in \tT(U),
\text{ then } s(x)y = s(xy).
\end{alignat*}
\end{remark}

\begin{proof}
$(\SC2) \Rightarrow (\SC2')$: A priori we have $e\s(x)y = xy \in
M$. By  $(\SC2)$ we know that either $\s(x)y = xy$ or $\s(x)y =
\s(xy). $ For  $s(x) y \in \tT(U)$, this implies  $\s(x)y = \s(xy).
$
\pSkip
$(\SC2') \Rightarrow (\SC2)$: Given $x\in M$, $y\in U$,  we
verify that $ \s(x)y  \in M \cup \s(M)$. If  $\s(x)y \notin M$,
then $(\SC2')$ states that  $\s(x)y = \s(xy) \in \s(M)$.
\end{proof}

\begin{thm}\label{thm:5.3} $ $
\begin{enumerate} \eroman
    \item If $\s$ is an ig-section in $U$, then the ideal
    compression (cf. \cite[Definition 2.5]{IKR4})
    $$ E:=E(\s) := E(U, \s(M)) = E(U, M \cup \s(M))$$ is an
    MFCE-relation such that each $E$-equivalence class  contains
    at most one tangible element, hence consists of at most 2
    elements. More precisely, the only equivalence classes, which
    are not one-element sets, are  classes of the form $[a]_E = \{a, \s(a)\}$
    with $a \in M$, $\s(a) \in \tT(U).$
    \item All these MFCE-relations $E(\s)$ are mixing.

    \item We have a bijection $\s \mapsto E(\s)$ from the set
    $\IGS(U)$ of all ig-sections in $U$ onto the set of all
    MFCE-relations $E$ on $U$ with $|[x]_E \cap \tT(U)| \leq 1$ for every $x \in U.$
\end{enumerate}

\end{thm}

\begin{proof} (i): Given an ig-section $\s:M \to U,$ let $x,y \in
U$, $x \neq y$ and $x \sim_{E(\s)} y.$ Then $ex=ey =:a$ and $x
\sim_{E(\s)} a \sim_{E(\s)} y.$ Since $\mfA(\s) \cap U_a = \{a,
\s(a) \}$, this forces $x =a$, $y = \s(a) \in \tT(U)$, or vice
versa. Thus $[x]_{E(\s)} = [a]_{E(\s)} = \{a, \s(a) \} $. \pSkip
(ii): Now obvious. \pSkip
(iii): Let $E$ be an MFCE-relation on $U$ such that each
equivalence class contains at most one  tangible element. Given
$a\in M$, we have
$$ [a]_E = \{a, \s(a) \} ,$$ where either $\s(a) = a$ or $\s(a)$ is  the
unique tangible element in $[a]_E$. This gives a  function $\s:
M \to U$ with $e \s(a) =a $ for all $a \in M$. Suppose $y \in U$
and $\s(a)y \in \tT(U)$. Then $\s(a) \in \tT(U)$ and $a \sim_E
\s(a)$. Since $E$ is multiplicative, it follows that $ay \sim_E
\s(a)y \in \tT(U)$, hence $\s(a)y = \s(ay)$. Claim (iii)
is now obvious.
\end{proof}

\begin{defn}\label{defn:5.4} We carry over the partial order on the
subposet $\{ E(\s) \ds | \s \in \IGS(U)\} $ of the lattice
$\MFC(U)$ of MFCE-relations on $U$ to the set $\IGS(U)$ by the
bijection $\s \mapsto E(\s).$ In other terms, for $s,t
\in \IGS(U)$ we define  $$t \leq s \dss \rightleftharpoons E(t) \subset
E(s).
$$
\end{defn}
Clearly the following holds.

\begin{remark}\label{rem:5.5} Let $s,t \in \IGS(U)$. The following
are equivalent.
\begin{enumerate} \eroman
    \item $t \leq s.$
    \item For any $a\in M$ either $t(a) = s(a)$ or $t(a) =a$.
\end{enumerate}
\end{remark}

 By  Theorem \ref{thm:5.3} it is obvious that $\{ E(\s) \ds | \s
\in \IGS(U)\} $ is a lower set of the  lattice $\MFC(U)$
contained  in  $\Mix(U)$, and that the \textbf{trivial ig-section}
$a \mapsto a$ is the bottom element of $\IGS(U).$
Moreover we
clearly have the following
\begin{prop}\label{prop:5.6} 
The infimum $s \wedge t $ of $s$ and $t$ exists in the poset $\IGS(U)$ and is
given~ by $$ (s \wedge t)(a) := \left\{
\begin{array}{ll}
  \s(a) & \text{if } \s(a) = t(a) \in \tT(U),\\
  a & \text{else}. \\
\end{array}
\right.
$$
Moreover $E(s \wedge t) = E(s) \cap  E(t).$
\end{prop}

In general, two elements $s,t$ of the poset $\IGS(U)$ need not
have a common upper bound, but if they do, there exists a least
upper bound (= supremum) of $s,t$. More precisely, we have the
following fact.

\begin{prop}\label{prop:5.7}
Assume that $s,t$ are ig-sections in $U$, and that there exists an
ig-section~$u$ in $U$ with $s \leq u$, $t \leq u$. Then the
function $s \vee t: M \to U $ given  by ($a \in M$)

$$ (s \vee t)(a) := \left\{
\begin{array}{ll}
  \s(a) = t(a)& \text{if } \s(a), t(a) \in \tT(U),\\[1mm]
  \s(a)       & \text{if } \s(a)  \in \tT(U), t(a) \in M, \\[1mm]
  t(a)       & \text{if } \s(a)  \in M, t(a) \in \tT(U), \\[1mm]
  a       & \text{if } \s(a)=  t(a)=a, \\
\end{array}
 \right.
$$
is a well defined ig-relation in $U$, and is the supremum  of $s$
and $t$ in the poset $\IGS(U). $ In the lattice $\MFC(U)$
$$E(s \vee t) = E(s) \vee E(t).$$
\end{prop}

\begin{proof}
Recall that for every  $a \in M$ both $\s(a)$ and $t(a)$ are
elements of the set $\{a, u(a) \}$. Thus the function $s \vee t$
is certainly well defined  (and independent of the choice of $u$).
It is routine  to verify  that $s \vee t$ is
an ig-section and is the supermum of $s,t$ in $\IGS(U)$. The
isomorphism $s \mapsto E(s)$ from $\IGS(U)$ onto a lower subset of
the lattice $\MFC(U)$ implies that $E(s) \vee E(t) = E(s \vee
t).$
\end{proof}

By the same vein we obtain.
\begin{prop}\label{prop:5.8} Assume that $(\s_i \ds | i \in I)$ is
a family in the poset $\IGS(U)$ such that for any two indices $i,j
\in I$ the elements $s_i, s_j$ have an upper bound in $\IGS(U)$
(hence the supremum $s_i \vee s_j$ exists by Proposition
\ref{prop:5.7}). Then we have a well defined function $s: M \to U$
with $s(a):= s_i (a)$ if there exists some $i \in I$ with $s_i(a )
\in \tT(U)$, and $s(a) : =a $ else. The function $s$ is an
ig-section in $U$, and is the supermum of the family $(\s_i \ds |
i \in I)$ in $\IGS(U)$, $s = \bigvee_i s_i$. In the complete
lattice $\MFC(U)$ we have
$$ E(s) = \bigvee_{i \in I} E(s_i).$$
\end{prop}
Applying this proposition to chains in $\IGS(U)$ we  obtain by
Zorn's lemma

\begin{cor}\label{cor:5.9}
For every $s \in \IGS(U)$ there exists a maximal element $s'$ of
the poset $\IGS(U)$ with $s \leq s'$.
\end{cor}

\begin{defn}\label{defn:5.10} We call an ig-section $s: M \to U$
\textbf{primitive}, if there exists some $x\in\tT(U)$ such that
the ideal $\mfA(s) = M \cup s(M)$ of $U$ is generated by $e$ and
$x$, $\mfA(s) = eU \cup xU$. In this case we call $x$ a
\textbf{generator} of the section $s$.
\end{defn}

In more elaborate term this means that $x = s(a) \in
\tT(U)$ with $a:= ex$, and if $z:= s(c) \in \tT(U)$ for some $c
\in M$, then there exists some $y \in U$ such that $z =xy$. As a
consequence, all products $xy$ with $y \in U$, $xy \in
\tT(U)$, $exy =c$, are equal. Indeed, we have $xy = s(a)y =
s(ay),$ hence $c = es(ay) = ay$, an thus  $xy = s(c).$

To get  a better grasp at this situation  we
introduce some terminology concerning   divisibility in the subset
$\tT(U)$ of the monoid $U$.

\begin{defn}\label{defn:5.11} Given $x,y \in \tT(U)$ we say that
$z$ is a \textbf{son} of $x$ and $x$ is a \textbf{father} of $z$,
if $z =xy$ with some $y \in U$ (which then necessarily lies in
$\tT(U)$). An element $x$ of $\tT(U)$ is said to be a
\textbf{tyrant} of $U$ (or ``$a$ is \textbf{tyrannic} in $U$''),
if for any $c \in M$ there exists at most one son $z$ of $x$ with
$ez =c.$
\end{defn}

The term   ``tyrant''  alludes to the property of $x$ that
every tangible ``genetic outcome''~ $xy$ of pairing $x$ with some
$y \in \tT(U)$ is completely determined by its ghost $exy$ and
~$x$.
Notice that the ghost $b:= ey$ is not uniquely determined
by $a := ex$ and $c$, since~$M$ may not be cancellative. \pSkip

In this terminology we can recast Definition \ref{defn:5.10} as
follows.

\begin{remark}\label{rem:5.12} A function $s: M \to U$ with $\nu_U \circ s =
\id_M$ is a primitive ig-section iff there exists a tyrant $x$ in
$U$ such that $s(M) \cap \tT(U)$ is the set of all sons of $x$.
These tyrants $x$ are then  the generators of the primitive
ig-section $s$.
\end{remark}

A primitive ig-section may have several generators, but they all
are ``associated'' in the following sense.
\begin{defn}\label{defn:5.13} We call two elements $x,z$ of
$\tT(U)$ \textbf{associated} (in $\tT(U)$) if $z$ is a son of~ $x$
and $x$ is a son of $z.$ We then write $x \sim _{\tT} z$ (or more
precisely  $x \sim _{\tT(U)} z$).
\end{defn}

We use similar terminology for divisibility in the monoid $M$.
\begin{defn} Given an elements $a,c$ of $M$, we say that $c$ is
\textbf{divisible}  by $a$, and write $a|c$, if $c = ab$ for some
$b \in M$ (perhaps not uniquely determined by $a$ and $c$). We say
that $a$ and~ $c$ are \textbf{associated in} $M$, and write $a
\sim_M c$, if $a | c$ and $c|a .$ (If necessary, we write $a |_M c $ instead of $a|c$.) We also set $a \sim_\tT c$
if $a \in \tT(U) c $  and $c \in \tT(U) a$.
\end{defn}

It is obvious from Definition \ref{defn:5.11} that any son of a
tyrant of $U$ is again a tyrant of $U$.
\begin{thm}\label{thm:5.14} Assume that $x$ is a tyrant of  $U$.
Let  $s_x: M \to U$  be the map defined by $s_x(c) := z$, if $x$ has a son $z$ with $ez =c, $ and $s_x(c):=c$ otherwise. Then $s_x$ is a
primitive ig-section of $U$ with generator $x$.
\end{thm}

\begin{proof} Proving that $s_x$ is an ig-section, it follows from Remark \ref{rem:5.12} that $s_x$ is
primitive with generator $x$.

Obviously $es_x(c) = c$ for every $c \in M$. Given $a \in M $ and
$y \in U$ with $s_x(a)y \in \tT(U)$, it remains to verify
that $s_x(a)y  = s_x(ay)$, cf. Remark \ref{rem:5.2}. Clearly
$s_x(a) \in \tT(U).$ Thus $z:=s_x(a)$ is a son of $x$, hence is
again a tyrant. Since $zy \in \tT(U), $ $zy$ is a son of $y$,
hence is  a son of $x$.  We conclude that $zy = s_x(c)$ with $c:=
ezy = ay$. On the other hand $zy = s_x(a)y$, and thus  $s_x(a)y
= s_x(ay).$
\end{proof}

\begin{thm}\label{thm:5.15} Assume that $s$ is an ig-section in
$U$. Then every $x \in s(M) \cap \tT(U)$ is a tyrant of $U$ and
$s_x \leq s, $ $s_x(a) = s(a) =x $ for $s:= ex$. More generally,
if $z$ is a son of $x$, then $s_x(ez) = s(ez).$
\end{thm}

\begin{proof} Let $x \in s(M) \cap \tT(U)$ and $a:= ex,$ hence $x =
s(a).$ Assume that $z$ is a son of $x$ and $c:= ez$.  Chose $y
\in U$ with $z = xy.$ Then $c = exy =ay$ and $z= s(a)y = s(ay)=
s(c),$ which  proves that $x$ has exactly one son $z$ with $ez =c$,
hence $x$ is a tyrant. Moreover $s(ez) = s_x(ez)$  for every son
$z$ of $x$. In particular $(z=x)$, $s_x(a) =s(a).$ Given $c \in M$, if $x$ has no son $z$ with $ez = c,$ then
$s_x(c)=c$. Thus $s_x(c) \in \{c, s(c) \}$ for every $c \in M$,
proving that $s_x \leq s.$
\end{proof}

Utilizing  Proposition \ref{prop:5.8},  we
conclude the following  from Theorem \ref{thm:5.15}.

\begin{cor}\label{cor:5.16} Let $s$ be a nontrivial ig-section in
$U.$ Then $\{ s_x \ds | x \in s(M) \cap \tT(U)\}$ is the set of
all primitive ig-sections $t \leq s$. Its  supremum in the poset $\IGS(U)$ is $s$.
\end{cor}
\{N.B. If $s$ is the trivial ig-section, this set is empty.\}
\begin{schol}\label{schl:5.17} Given an element $x$ of $\tT(U)$,
the following are equivalent.
\begin{enumerate}   \eroman
    \item  $x$ is tyrant of $U$,
    \item  there exists an ig-section  $s: M \to U$ with $x \in s(M),$
    \item $x$ is a generator of a primitive ig-section in $U$.
\end{enumerate}
\end{schol}

\begin{proof}
$(iii) \Rightarrow (ii)$: trivial.   $(ii) \Rightarrow (i)$: immediate  by
Theorem~\ref{thm:5.15}. $(i) \Rightarrow (iii)$: clear by Remark~
\ref{rem:5.12}.
\end{proof}

We study the effect of a fiber contraction on ig-sections and
tyrants.
\begin{prop}\label{prop:5.18} Assume that $\al: U \onto V$ is a
fiber contraction over $M = eU = eV.$
\begin{enumerate} \eroman
    \item If $s: M \to U$ is an ig-section on $U$, then $\al \circ
    s$ is an ig-section on $V$.

    \item Let $s: M \to U$ be a primitive ig-section with
    generator $x$. If $\al(x)$ is tangible, then $\al \circ s$ is
    primitive with generator $\al(x)$. Otherwise $\al \circ s$ is
    trivial.

    \item If $x$ is a tyrant of $U$ and $\al(x)$ is tangible, then
    $ \al(x)$ is a tyrant of $V$ and $s_{\al(x)} = \al \circ s_x.$
\end{enumerate}
\end{prop}

\begin{proof} (i): We verify  conditions $(\SC 1)$ and $(\SC 2)$ in
Definition \ref{defn:5.1} for the section  $s' := \al \circ s$.
Clearly $$\nu_V \circ s' = \nu _V \circ \al \circ s = \nu_U \circ
s  = \id_M.$$ Furthermore, we  have $s(M)U \subset M \cup s(M).$
Applying $\al$ we obtain $s'(M) V \subset M \cup s'(M).$
\pSkip
(ii): We work directly with Definition \ref{defn:5.10}. Assume that
$s$ is  primitive with generator ~$x$. Then $M \cup s(M) = M \cup
xU$. Applying $\al$ we obtain
$$ M \cup (\al \circ s) (M) = M \cup \al(x)V.$$
\pSkip
(iii): Let $x $ be a tyrant of $U$ and $s:= s_x.$ As just proved
$\al \circ s$ is primitive with generator~ $\al(x).$ By Scholium
\ref{schl:5.17} this means that $\al(x)$ is a tyrant of $V$ and
$\al \circ s = s_{\al(x)}.$
\end{proof}

It may happen that the supertropical monoid $U$ has no
non-trivial ig-sections,  then $U$  contains no tyrannic
elements. However, in
good cases, there exists a canonical way to produce tyrants  by dividing out tangible MFCE-relations in $U$, to
be explained in~\S\ref{sec:7}.

\section{Equalizers}\label{sec:6}

Let $U$ is a supertropical semiring  with $M := eU$, and let $S$
be an arbitrary  subset of $U$. We look for \MFCE-relations $E$ on $U$
such that $s \sim_E t$ for any two $s,t \in S$ with $es =et.$ The
subset $\mfM_S$ of $\MFC(U)$ consisting  of these relations is
certainly not empty, since it contains $E(\nu_U).$ Thus there
exists a finest such relation $E$, namely the element $\bigwedge
\mfM_S$ in the complete lattice $\MFC(U). $

\begin{defn}\label{def:6.1} If $S \subset U_a = \{x \in U \ds | ex = a\}
$ for some $a \in M$, we call $E$ the \textbf{equalizer} of the
set $S$, and write $E = \Eq(S).$ In general we call $E$ the
\textbf{fiberwise equalizer} of $S$, and write $E = \Feq(S).$
\end{defn}

Recall that for any $X \subset U$ and $c
\in M$ we write
$$ X_c := X \cap U_c = \{ x \in X \ds | ex =c\}.$$

\begin{remark}\label{rem:6.2} It is obvious from the definition that
$$ \Feq(S) = \bigvee_{c\in M } \Eq(S_c).$$
\end{remark}

Here and in all the following we do not exclude the case that $S$
(or some $S_c$) is empty. If $S = \emptyset$ then by definition
$\Feq(S) = \diag(U).$ This also holds if $S$ is a one-point set.
\pSkip

We strive for an explicit description of  $\Feq(S)$. First an easy
case.

\begin{prop}\label{prop:6.2} Assume that $a \in S \subset U_a.$
Then $\Eq(S)$ is the relation $E(U,S) := E(U, M \cup US)$ which
compresses the ideal  $M \cup US$ to ghosts (cf. \cite[Corollary
1.17]{IKR4}).
\end{prop}

\begin{proof} Making the elements of $S$ equal means identifying
them with their common ghost companion $a$, then every element
of $US$ is identified with its ghost companion.
\end{proof}

\begin{defn}\label{def:6.3} An \textbf{$S$-path} $\gm$ of  \textbf{length} in $U$ is a
finite sequence $z_1, \dots, z_{n+1}$ of elements of $U$, called
the \textbf{nodes} of $\gm$, together with a sequence of triples
\begin{equation}\label{eq:3.1}
    (s_k, u_k, t_k) \in S \times U \times S, \qquad (1 \leq k \leq
    n ),
\end{equation}
called the \textbf{labels}  of $\gm$, such that $es_k = et_k$ for
$1 \leq k \leq n$ and
\begin{equation}\label{eq:3.2}
\begin{array}{lll}
  u_k s_k = z_k &   &  (1 \leq k \leq n), \\[2mm]
  u_k t_k = z_{k+1} &   &  (1 \leq k \leq n). \\
\end{array}\end{equation}
We say that  $\gm$ is an
\textbf{elementary $S$-path}, if $n=1$. Given $z,w \in U$,
$\gm$ an \textbf{$S$-path connecting $z$ to $w$} (or:
\textbf{from $z$ to $w$}) if $z =z_1, w =z_{n+1}$.
\end{defn}

Notice that, given  an $S$-path $\gm$, every ``inner'' node
$z_k$ ($2 \leq k \leq n$) is presented as a product of  an element
of $U$ and an element of $S$ in two different ways,
\begin{equation}\label{eq:3.3}
z_k = u_k s_k = u_{k-1} t_{k-1}.\end{equation} Notice also that
$ez_1 = \cdots = e z_{n+1}.$
Observe finally that our notation of an $S$-path~ $\gm$ contains a
redundancy: The sequence of labels determines the sequence of
nodes.

In the following it may be help to visualize $S$-paths by diagrams, where multiplication by an element $u \in \tT(U)$ is presented by an arrow $\overset{u}{\To}$. For example, the diagram of an $S$-path~$\gm$ with $3$ nodes ($n=2$).

$$ \hskip -2in
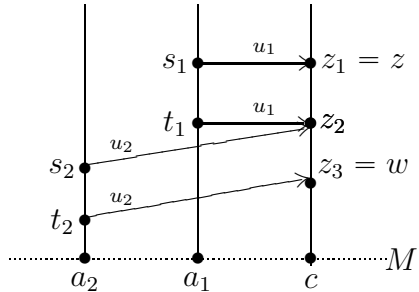
\xy <1cm,0cm>:
(0,0)*=0{}="+" ;
(5,0)*=0{\quad M}="*" **@{.},
(1,-0.3)*+{a_2}; (1,3.5)*+{} **@{-}
?!{"+";"*"} *{\bullet}
,
(2.5,-0.3)*+{a_1}; (2.5,3.5)*+{} **@{-}
?!{"+";"*"} *{\bullet}
,
(4,-0.3)*+{c}; (4,3.5)*+{} **@{-}
?!{"+";"*"} *{\bullet}
,
(1,0.5)*=0{\bullet}="+";
(4,1)*=0{\bullet}="*";
(0.7,0.5)*+{t_2}; (4.7,1.2)*+{z_3=w} **{}?>*@2{>}**@{-}
?!{"+";"*"} *{},(1.5,0.8)*{_{u_2}},
,
(1,1.2)*=0{\bullet}="+";
(4,1.8)*=0{\bullet}="*";
(0.7,1.2)*+{s_2}; (4.3,1.8)*+{z_2} **{}?>*@2{>}**@{-}
?!{"+";"*"} *{},(1.5,1.5)*{_{u_2}}
,
(2.5,1.8)*=0{\bullet}="+";
(4,1.8)*=0{\bullet}="*";
(2.2,1.8)*+{t_1}; (4.3,1.8)*+{z_2} **{}?>*@2{>}**@{-}
?!{"+";"*"} *{},(3.4,2)*{_{u_1}}
,
(2.5,2.6)*=0{\bullet}="+";
(4,2.6)*=0{\bullet}="*";
(2.2,2.6)*+{s_1};  (4.7,2.6)*+{z_1=z}  **{}?>*@2{>}**@{-}
?!{"+";"*"} *{}, (3.4,2.8)*{_{u_1}}
\endxy
$$

\begin{defn}\label{def:6.4}
Let $\gm$ be an $S$-path  from $z$ to $w$ with sequence of
labels \eqref{eq:3.1}. We obtain the \textbf{inverse $S$-path
$\igm$}  from $w$ to $z$ by changing the sequence of labels
\eqref{eq:3.1} to
\begin{equation*}
    (t_{n-k}, u_{n-k}, s_{n-k}) \in S \times U \times S, \qquad (0 \leq k \leq
    n-1 ).
\end{equation*}
Furthermore, if $\gm'$ is an $S$-path from $w$ to $w'$ we obtain a
path $\gm * \gm'$  from $z$ to $w'$ by juxtaposing  the
sequence of labels of $\gm'$ to the sequence of labels of $\gm$.
\end{defn}

\begin{notations}\label{notations:6.5} Henceforth we denote an
elementary $S$-path simply by its labels, and an  $S$-path $\gm$
with sequence of labels \eqref{eq:3.1} as
$$ \gm = (s_{1}, u_{1}, t_{1}) * \cdots * (s_{n}, u_{n}, t_{n}).$$
In this notation we have
$$ \igm = (t_{n}, u_{n}, s_{n}) * \cdots * (t_{1}, u_{1},
s_{1}).$$
\end{notations}

For later use we quote the following obvious fact.
\begin{rem}\label{rem:6.6} Given an $S$-path $\gm$ with sequence of
labels \eqref{eq:3.1} and an element $u' $ of $U$, we obtain a new
$S$-path
$$ u' \gm = (s_{1}, u' u_{1}, t_{1}) * \cdots * (s_{n}, u' u_{n}, t_{n}).$$
If $\gm$ connects $z$ to $w$ then $u'\gm$ connects  $u'z$ to
$u'w$.
\end{rem}

Given an element $z_0$ of the monoid ideal
$$ US := \{ us \ds | u \in U,  s \in  S \} $$
generated by $S$, and choosing $u_0 \in U$, $s_0 \in S$ with $z_0
= u_0 s_0$, we get the elementary path  $(s_0, u_0, s_0)$. When $z_0 \notin US, $ there is no $S$-path
which starts or ends at~ $z_0.$

\begin{defn}\label{def:6.7} Two elements $z,w$ of $US$ are called
\textbf{$S$-connected},  written $z \sim_S w,$  if there
exists an $S$-path  from $z$ to $w.$
\end{defn}

By the above discussion it is evident that
``$S$-connected'' is an equivalence  relation on the set $US$, whose
equivalence classes are called  the \textbf{$S$-components} of $US$
(or: of~$U$).

\begin{thm}\label{thm:6.8}
The equivalence classes of the MFCE-relation $\Feq(S)$ are the
$S$-components of $US$ and the one-element sets $\{ z \}$ with $z
\in U \sm \{ US\}.$
\end{thm}

\begin{proof}
We define an equivalence relation $E$ on $U$ by $z \sim_E
w$  iff either $z= w$ or there exists an $S$-path from  $z$
to $w$. We learn from Remark \ref{rem:6.6} that $z \sim_E w$
implies $uz \sim_E uw$ for any $u \in U.$ As observed above, every
$S$-path runs in a fiber $U_c,$
$c \in M$. 
 Thus  $z \sim_E w$ implies also $ez
= ew$. This proves that $E$ is an MFCE-relation. If $s,t \in S,$
the elementary $S$-path $(s,1,t)$  runs form $s$ to $t$, hence $s
\sim_E t.$

Assume  that $F$ is any MFCE-relation on $U$ with $s \sim_F t$
for all $s,t \in S.$ Given an elementary $S$-path $(s,u,t)$ we
conclude form  $s \sim_F t$ that  $u s  \sim_F u t.$ It follows
that any two $S$-connectable elements of $US$ are $F$-equivalent,
hence $F$ is coarser than $E$, proving  that $E = \Feq(S). $
\end{proof}

These arguments give us a constructive proof for the
existence of the fiberwise equalizer $\Feq(S)$,  which does not use  the fact
that $\MFC(U)$ is a complete lattice.
\begin{cor}\label{cor:6.9} Let  $S$ be any subset of
$U$. The following are equivalent.
\begin{enumerate}\eroman
\item $\Feq(S)$  is ghost separating.

\item The nodes of each  $S$-path are in one of the sets
$\tT(U)$, $\tG(U)$, $\00.$

\item For every elementary $S$-path $(s,u,t)$  both $us$ and $ut$
are contained in one of the sets $\tT(U)$, $\tG(U)$, $\00.$

\item For every $u \in U$ and $c \in M$ either $uS_c \subset \tT(U)$, or $uS_c \subset
\tG(U)$, or $uS_c = \00.$
\end{enumerate}
\end{cor}

\begin{proof} (i) $\Leftrightarrow$ (ii): This is clear from
Theorem \ref{thm:6.8}.
\pSkip
(ii) $\Leftrightarrow$ (iii): Obvious.
\pSkip
(iii) $\Rightarrow$ (iv): Given $u \in U$, assume that $S_c
\neq \emptyset$, and fix some $s_0 \in S_c.$ Apply (iii) to the
elementary path $(s_0,u, s)$ with $s$ running through $S_c.$ The
element $u s_0$ is contained in one of the sets $\tT(U)$,
$\tG(U)$, $\00$. It follows from (iii) that $uS_c$ is contained in
the same set.
\pSkip
(iv) $\Rightarrow$ (iii): Fix an elementary $S$-path $(s,u,t),$ and let  $c := es =et.$ By (iv), the set $uS_c$ is contained in one of
the sets $\tT(U)$, $\tG(U)$, $\00,$ implying  that $\{ us, ut\} $
is a subset of the same set.
\end{proof}

In a similar way we handle MFCE-relations which equalize fiberwise
all members  of a family of subsets of $U$ instead of a single
subset of $U$. The following is obvious.

\begin{prop}\label{prop:6.11} Let $\mfS := (S_\lm \ds | \lm \in
\Lm)$ be a family of subsets of $U$. Then
$$ \Feq(\mfS) := \bigvee_{\lm \in \Lm} \Feq(S_\lm)$$
is the finest MFCE-relations $E$ on $U$ such that, if $s,t \in
S_\lm $ for some $\lm \in \Lm$ and $es= et $, then $s \sim_E t.$
\end{prop}
This relation $ \Feq(\mfS) $ can be described by ``paths" in a way
analogous to Theorem \ref{thm:6.8}.

\begin{defn}\label{defn:6.12} An \textbf{elementary $\mfS$-path} is
a triple $(s,u,t) \in U \times U \times U$ with $es= et$ and
$\{s,t  \} \in S_\lm$ for  some $\lm \in \Lm.$ Consequentially an
$\mfS$-path $\gm$ is a sequence of such triples $(s_k,u_k,t_k)$,
 $ 1 \leq k \leq n$,  written
 $$ \gm =  (s_1,u_1,t_1) * \cdots *  (s_n,u_n,t_n),$$
 such that the equations  \eqref{eq:3.2} from above are valid.
 Again we call the elements
 $$ z_1 = u_1 s_1, \quad  z_{n+1} = u_n t_{n}, \quad
 z_k = u_k s_k =  u_{k-1} t_{k-1} \ (2 \leq k \leq n),$$
the \textbf{nodes}  of $\gm$, and say that $\gm$ runs from $z_1$
to $z_{n+1}.$
\end{defn}

Notice that all nodes $z_k$ are elements of the monoid ideal
$\bigcup_{\lm \in \Lm} U S_\lm.$ In particular, if $z_0 \in U$ is
not in this set, there exists no $\mfS$-path starting or
ending at $z_0.$

\begin{thm}\label{thm:6.13} If $z,w \in U$  and $z \neq w$, then $z$
and $w$ are $\Feq(\mfS)$-equivalent iff there exists an
$\mfS$-path  form $z$ to $w.$
\end{thm}

\begin{proof} Since $\Feq(\mfS) = \bigvee _{\lm \in \Lm} \Feq
(S_\lm)$ is the equivalence relation on $U$ generated by the
relations $\Feq(S_\lm)$ (cf. \cite[\S7]{IKR1}), the given elements
$z \neq w$ are $\Feq(\mfS)$-equivalent iff there exists a sequence
$z = x_1, x_2, \dots, x_N = w$ in $U$ and indices $\lm(1), \dots,
\lm(N-1)$ in $\Lm$ such that $x_i \neq x_{i+1}$ for $1 \leq i \leq
N-1$ and
 \begin{equation}\renewcommand{\theequation}{$*$}\addtocounter{equation}{-1}\label{eq:str.1}
x_i \ds{\sim_{\Feq(S_{\lm(i)})} } x_{i+1} \qquad (1 \leq i \leq
N-1).
 \end{equation}
 If this holds, then by Theorem \ref{thm:6.8} there exists an
 $S_{\lm(i)}$-path $\gm_i$  form $x_i$ to $x_{i+1}$, and $\gm : = \gm_1 *
 \cdots *
 \gm_{N}$ is an $\mfS$-path from $z$ to $w.$

 Conversely, given an $\mfS$-path $\gm$ from $z$ to $w$ with
 nodes $z_1, \dots, z_{n+1}$, write
$\gm : = \gm_1 * \cdots  *
 \gm_{n}$, where each  $\gm_k$ is elementary from $z_k$ to $z_{k+1}.$
Omitting the $\gm_k$ with $z_k = z_{k+1}$, we  are in the
situation $(*)$, with $N = n+1$, $x_i = z_i.$
\end{proof}

In complete analogy to the proof of Corollary \ref{cor:6.9} we
obtain
\begin{cor}\label{cor:6.14}
Let $\mfS$ be a family of subsets of $U$. The following are
equivalent:
\begin{enumerate} \eroman
    \item $\Feq (\mfS)$ is ghost separating.

    \item The nodes of each $\mfS$-path are  in one of the sets
    $\tT(U)$, $\tG(U),$ $\00$.

    \item For every $u \in U,$ $c \in M,$ $\lm \in \Lm$ the set
    $u(S_\lm)_c$  is contained in one of the sets  $\tT(U)$, $\tG(U),$
    $\00$.
\end{enumerate}

\end{cor}

\section{Producing isolated tangibles and tyrants}\label{sec:7}
As before $U$ is a supertropical semiring and $M:= eU$.
\begin{notations}\label{notation:7.1}  $ $
\begin{enumerate} \ealph
    \item Given $x,z \in U,$ we say that $x$ \textbf{divides} $z$
    (in $U$), written $x | z$, if $z \in Ux.$ Otherwise we write
    $x \nmid z$. (N.B. If $a,c \in M$, then divisibly of $c$ by $a$
    means the same in $M$ as in~ $U$, since $c \in Ua$ iff $c \in Ma$.)
    \item We  define
    $$ [z:x]_U := [z:x] = \{ u\in U \ds | ux =z\}. $$
    Of course, if $x \nmid z$  then $[z:x] = \emptyset.$

    \item Given $x \in \tT(U),$ we define
    $$ \SS(x) := (Ux) \cap \tT(U),$$
    and call, as already done in \S\ref{sec:2}, the elements
    of $\SS(x)$ the \textbf{sons} of $x$. The sons $z \neq x $ are said to be \textbf{proper sons} of  $x$.  We let
    $$\SS_c(x):= \SS(x)_c = (Ux)_c \cap \tT(U),$$
and call the elements of $\SS_c(x)$ the \textbf{sons of $x$ over
$c$}.

\item We denote the fiberwise equalizer $\Feq(\SS(x))$ (cf.
\S\ref{sec:3}) by $\T(x)$ and the equalizer $\Eq(\SS_c(x))$ by
$\T_c(x).$
\end{enumerate}
\end{notations}

We explicate the meaning of these MFCE-relations, first of
$\T_c(x)$ and then of $\T(x).$ Studying $\T_c(x)$ we may assume that $c
\in \tT(U)a,$ since otherwise $\SS_c(x)$ is empty. The following
remains true  without  this assumption.

\begin{prop}\label{prop:7.2} Let $x \in \tT(U)$ and $c \in M.$
\begin{enumerate} \ealph
    \item $\T_c(x)$ is the finest MFCE-relation $E$ on $U$ such that
    $[x]_E$ is ghost or $[x]_E$ is tangible and has at most one son
    over $c$ in $U/ E$. (Recall that we  identify $M/E = M.$)

    \item More explicitly we have the following alternative.

\begin{description}
    \item[Case I] $\T_c(x)$ is the finest MFCE-relation $E$ on $U$
    such that $[x]_E$ is tangible and has exactly one son over
    $c$ in $U/ E.$
    \item[Case II] There exists no such relation $E$. Then $\T_c(x)= E(U,\SS_c(x)).$
\end{description}
\end{enumerate}\end{prop}

\begin{proof} These claims are essentially trivial. First observe
that for any MFCE-relation $E$ on~ $U$ we have
$$ ([x]_E \cdot (U/E) )_c = \{ [ux]_E \ds | u \in U, \, eux=c \} = \{ [z]_E \ds | z \in \SS_c(U)\} \cup \{ c \}. $$
Thus $\T_c(x) = \Eq(\SS_c(x))$ is the finest MFCE-relation $E$ such
that either $[x]_E$ is ghost, then $[z]_E= c$ for all $z \in
\SS_c(x)$), or $[x]_E$ is tangible and has at most one son over $c$.
Case~ II happens iff $\T_c(x)$ identifies all elements of $\SS_c(x)$
with the ghost $c$. Proposition \ref{prop:6.2} gives
$\T_c(x) = E(U, \SS_c(x)).$
\end{proof}
By analogous arguments we obtain
\begin{prop}\label{prop:7.4}
Let $x \in \tT(U)$.
\begin{enumerate} \ealph
    \item $\T(x)$ is the finest MFCE-relation $E$ on $U$, such
    that $[x]_E$ is a ghost, or
    $[x]_E$ is a tyrant in $U/E $ (i.e.,  $[x]_E$ is tangible, and for any $c\in M$ has  at most one son
    over ~$c$, cf.~
    \S\ref{sec:2}).

    \item More explicitly we have the following cases.

\begin{description}
    \item[Case I] $\T(x)$ is the finest MFCE-relation $E$ on $U$
    such that $[x]_E$ is a tyrant in~$U/ E.$
    \item[Case II] There exists no such relation. Then $\T(x)= E(U,\SS(x)).$
\end{description}
\end{enumerate}
\end{prop}


\begin{thm}\label{thm:7.4} Let $x \in \tT(U),$ $a := ex.$ The
following are equivalent.
\begin{enumerate}\eroman
    \item $[x]_{\T(x)}$ is ghost in $U/\T(x).$

    \item There exists an elementary $\SS(x)$-path $(vx,u,wx)$
     from a son $z= uvx$ of $x$ over $a$ to $ez = a.$

    \item There exist elements $u,v,w$ of $U$ with $evx = ewx,$ $euvx =
    ex,$ $wx \in \tT(U),$ $uvx \in \tT(U),$ $uwx \in M.$
\end{enumerate}
\end{thm}

\begin{proof} (i) $\Rightarrow$ (ii): (i) means that there exists
an $\SS(x)$-path
$$ \gm = (v_1x, u_1,w_1 x) * \cdots * (v_nx, u_n,w_n x) $$
 from $x$ to $ex$ (cf. Theorem \ref{thm:6.8}). We choose
$\gm$ of minimal length. Then $u_n v_n  x \in \tT(U)$, $u_n w_n x
\in M.$ We have $u_1 v_1 x = x,$ hence
$$ e u_i v_i x = e u_i w_i x = ex = a$$
for each $i \in \{ 1, \dots, n\}.$ Thus  $z:= u_n v_n x$ is a son
of $x$ over $a$, and $(v_n x, u_n, w_n x)$ is an elementary
$\SS(x)$-path  from $z$ to $ez = a$.
 \pSkip
(ii) $\Rightarrow$ (i): $\gm:= (x,1,z) * (vx,u ,wx)$ is an
$\SS(x)$-path  form $x$ to $a.$ Thus $[x]_{\T(x)} =
[ex]_{\T(x)}$ $(=a)$ is ghost.
\pSkip
(ii) $\Leftrightarrow$ (iii): Condition (iii) means that
$(vx,u,wx)$ is an $\SS(x)$-path from $z = uvx$ to $ez$, and $z$ is a
son of $x$ over $a.$
\end{proof}

\begin{schol}\label{schol:7.5} We can reformulate the equivalence
(i) $\Leftrightarrow$ (ii) in Theorem \ref{thm:7.4} as:
Given $x \in \tT(U),$ the element $[x]_{\T(x)}$ is tangible, hence
a tyrant, iff for all $u,v,w \in U$ with $evx = ewx$, $euvx = ex,$
and both $wx$ and $uvx$   tangible, also the product $uwx$ is
tangible.
\end{schol}

Exploiting the contents of Corollary \ref{cor:6.9} for $\SS = \SS(x)$,
we obtain a criterion that $[z]_{\T(x)}$ is tangible for every son
$z$ of $x$.

\begin{thm}\label{thm:7.6}
 Let $x \in \tT(U).$  The
following are equivalent.
\begin{enumerate} \eroman
    \item The MFCE-relation $\T(x)$ is tangible.

    \item For every son $z$ of $x$ in $U$ the element $[z]_{\T(x)}$
    is tangible in $U/ \T(x)$ (and hence a tyrant).

    \item  If $u,v,w \in U$ and  $evx = ewx,$  $uvx \in \tT(U),$ $wx \in \tT(U),$ then
    $uwx \in \tT(U).$
\end{enumerate}
\end{thm}

\begin{proof}
(i) $\Leftrightarrow$ (ii): This is a consequence of a general
fact for TE-relations. If $E$ is a TE-relation on $U$ and $z \in
U$ then $[z]_E$ is ghost iff $[z]_E = e [z]_E = [ez]_E$, i.e., $z
\sim_E ez$. Thus $[z]_E$ is tangible for every $z \in \tT(U)$ iff
$z \nsim_E ez$ for every $z \in \tT(U),$ which means that $E$ is
tangible (= ghost separating) when $E$ is MFCE. In the present
case, where $E =\T(x),$ we may focus on the elements $z \in \SS(x),$
since for $z \notin \SS(x)$ we have $[z]_E = \{ z\},$ hence
 $z \nsim_E ez$ for $z$ tangible.

\pSkip
(ii) $\Leftrightarrow$ (iii): By (i) $\Leftrightarrow$ (iii) of
Corollary \ref{cor:6.9} we know that $\T(x)$ is ghost separating
iff, for any $u,v,w \in U$ with $vx, wx \in \tT(U),$ $evx = ewx,$
the set $\{ uvx, uwx\}$ is contained in $\tT(U)$ or in $M$. Now,
if $ux \in M,$ then trivially $\{ uvx, uwx\} \subset M.$ Thus we
may focus on elements~$u$ with $ux\in \tT(U)$, and see that
condition (iii) in Corollary \ref{cor:6.9} translates to condition
(iii)  in the present claim.
\end{proof}

\begin{examp}\label{examp:7.7}
Let $x \in \tT(U).$ Assume that the product of any two sons of $x$
in $U$ is tangible. Then, for every son $z$ of $x$,  $[z]_{\T(x)}$
is a tyrant in $U/\T(x)$. Indeed,  we  conclude from $uvx \in
\tT(U)$ and $wx \in \tT(U),$ that $uvx \cdot wx \in \tT(U),$ hence
$uwx \in \tT(U).$
\end{examp}
\pSkip

We turn  to a property of tangible elements, which is weaker
than being a tyrant, but seems to have equal
importance.

\begin{defn}\label{defn:7.8}
Let $x \in \tT(U)$ and $a:= ex.$
\begin{enumerate} \ealph
    \item We call $x$ \textbf{isolated} (in $U$) if $\SS_a(x) = \{
    x\}$, i.e., $x$ has no proper son over $a$.

    \item We call $\T_a(x) = \Eq(\SS_a(x))$ the \textbf{isolating
    MFCE-relation} (on $U$) \textbf{for $x$} and  denote it by $\Is(x).$
\end{enumerate}
\end{defn}

Proposition \ref{prop:7.2} tells us that $\Is(x)$ is the finest
MFCE-relation $E$ on $U$, such that either $[x]_E \in \tT(U/E)$
and $[x]_E$ is isolated (Case I), or $[x]_E$ is ghost (Case II).
When the former case holds, we say that \textbf{$x$ can be
isolated} (in $U$), otherwise we say that~$x$ \textbf{cannot be
isolated}. In
the latter case  $\Is(x) = E(U,x)$.
\begin{thm}\label{thm:7.9}
Let $x \in \tT(U)$ and $a:= ex.$ The following are equivalent.
\begin{enumerate} \eroman
    \item $x$ cannot be isolated in $U$.

    \item There exist $u,v,w \in [a:a]_U$ such that
    $uvx \in \tT(U),$  $wx \in \tT(U),$ but $uwx \in M$ (hence
    $uwx=a$).
\end{enumerate}

\end{thm}

\begin{proof} (ii) $\Rightarrow$ (i): Let $E := \Is(x) =
\Eq(\SS_a(x))$. Since $uvx, vx, wx \in \SS_a(x)$, we conclude~ that $$ x
\ds{\sim_E} uvx \ds{\sim_E} uwx \ds{\sim_E} a.$$
\pSkip
(i) $\Rightarrow$ (ii): There exists an $\SS_a(x)$-path
$$ \gm = (v_1x, u_1, w_1 x) * \cdots * (v_nx, u_n, w_n x)$$
of shortest length $n$ connecting $x$ to $a$.  Then $u_n v_n x \in
\tT(U)$ and $u_n w_n x =a.$ The elements $v_i x$, $w_ix$ lie in
$\SS_a(x)$ and $x = u_1 v_1 x \in \SS_a(x).$ It follows that $a = v_i
a = w_i a,$ i.e., $v_i, w_i \in [a:a].$ Furthermore
$$ a \ds = u_1 v_1 a \ds = u_1 w_1 a \ds = \cdots \ds = u_n w_n a, $$
and thus  all $u_i \in [a:a].$ Condition (ii) holds
with $u = u_n, $  $v = v_n$, $w= w_n.$
\end{proof}

\begin{examp}\label{examp:7.10} Let $a \in M \sm \00.$ Assume that
the set $[a:a]_U \cap \tT(U)$ is closed under multiplication. Then
condition (ii) in Theorem \ref{thm:7.9} is violated for every $x
\in \tT(U)_a.$ Thus every $x \in \tT(U)_a$ can be isolated.
\end{examp}

It turns out that the isolating relation $\Is(x)$ does not alter
if we replace $x$ by an associated element (cf. Definition
\ref{defn:5.13}).

\begin{lem}\label{lem:7.11} Assume that $x, x' \in \tT(U)$ are
associated, i.e., $x | x'$ and $x' | x.$ Then, for any~ $u \in U$,
\begin{enumerate} \eroman
    \item $ux \in \tT(U) \Leftrightarrow ux' \in \tT(U),$
    \item $eux = ex \Leftrightarrow eux' = ex'.$
\end{enumerate}
\end{lem}

\begin{proof}
In both claims it suffices to verify the direction $\Rightarrow.$
Given elements $v,v'$ of $U$ with $x' = vx,$ $x = v' x'$, if $ux
\in \tT(U),$ then $uv'x' \in \tT(U),$ hence $ux' \in \tT(U).$ If
$eux = ex,$ then
$$ eux' \ds = euvx \ds = eux \cdot v \ds = exv \ds = ex'.$$ \vskip -7mm
\end{proof}

\begin{prop}\label{prop:7.12} Assume that $x' \in \tT(U)$ is
associated to $x \in \tT(U).$ Then
\begin{enumerate} \eroman
    \item  $\Is(x) = \Is(x'),$
    \item In particular, $x$ is isolated iff $x'$ is isolated,
    \item $x$ can be isolated iff $x'$ can be isolated.
\end{enumerate}
\end{prop}

\begin{proof} Choose again $v,v' \in U$ with
$ x' = vx,$ $x = v' x'$, and let $a:= ex,$ $a' := ex'.$

\pSkip
(i): We learn from Lemma \ref{lem:7.11}  that  $\SS_{a'}(x') = v
\SS_a(x').$ Thus  $\Eq(\SS_a(x)) \supset \Eq(\SS_{a'}(x')).$
By symmetry these  equivalence relations are equal,  i.e., $\Is(x) =
\Is (x').$
\pSkip
(ii): In this situation
$ \Is(x) = \Is (x') = \diag(U).$
\pSkip
(iii): Let $E:=  \Is(x) = \Is (x').$ From $x | x'$ and $x' | x$ we
conclude that $[x]_E \ds | [x']_E$ and $[x']_E \ds  | [x]_E$. Thus
$[x]_E $ is tangible iff $[x']_E$ is tangible.
\end{proof}

  We are interested in cases where a tangible element and all its
sons are simultaneously isolated.

\begin{defn}\label{defn:7.13} Let $x\in \tT(U)$. We  write
$$ \Sis(x) := \bigvee_{z\in \SS(x)} \Is(z), $$
and call it  the \textbf{son isolating
MFCE-relation for $x$} (on $U$).
\end{defn}

In other terms, $\Sis$ is the fiberwise equalizer $\Feq(\mfS(x))$
of the family
$$ \mfS(x) := \{ \SS_c(z) \ds | z \in \SS(x), c = ez\} $$
of subsets of $U.$  This is the finest MFCE-relation $E$ on $U$
such that for every son $z$ of $x$ the element $[z]_E$ is either
ghost or isolated tangible in $U/ E$. 

\begin{rem}\label{defn:7.14}
We have the following cases.
\begin{description}\dispace
    \item[Case I] $[x]_{\Sis(x)}$ is tangible. Then all sons of
    $[x]_{\Sis(x)}$  in  $U / {\Sis(x)}$ are isolated.

    \item[Case II]  $[x]_{\Sis(x)}$ is ghost. Then
    ${\Sis(x)} = E(U,x).$
\end{description}
\end{rem}

\begin{notations}\label{notation:7.15} $ $
\begin{enumerate} \ealph
    \item If $z \in \tT(U),$ $c:= ez$, we call an $\SS_c(z)$-path
    also an
    \textbf{$\Is(z)$-path}. Thus, if $z$ is a son of $ x \in
    \tT(U)$, $z = px$, an elementary  $\Is(z)$-path is a triple
    $(vpx,u,
    wpx)$ with $u,v,w,p \in U,$ $vpx \in \tT(U),$ $wpx \in
    \tT(U),$ and $evpx = ewpx = epx.$
    \item If $x \in \tT(U)$, we call an $\mfS(x)$-path (cf. Definition \ref{defn:6.12}) also
    an \textbf{$\Sis(x)$-path}.  Thus, an elementary  $\Sis(x)$-path is an
elementary $\Is(z)$-path for some son $z = px$ of $x$.
\end{enumerate}
\end{notations}

As in the proof of  Theorem  \ref{thm:6.13}, but
replacing $\mfS$-paths by $\mfS(x)$-paths, we obtain

\begin{thm}\label{thm:7.16}
Let $x \in \tT(U)$ and $a:= ex.$ The following are equivalent.
\begin{enumerate} \eroman
    \item $[x]_{\Sis(x)}$  is ghost in   $U / {\Sis(x)}$.

    \item There exists a son $z$ of $x$ and an elementary
    $\Is(z)$-path $(vz,u, wz)$ from a son $uvz$ of~$z$ over~$a$ to
    $a$. (N.B. $vz,wz \in \tT(U)$.)

    \item There exists  $u,v,w,p \in U$ with $epx = evpx = ewpx,$
    $euvpx = a$, $uvpx \in \tT(U)$, $wpx \in \tT(U)$, $uwpx= a.$
\end{enumerate}
\end{thm}

Arguing  again  as in the proof of  Theorem  \ref{thm:6.13},
we obtain

\begin{thm}\label{thm:7.17}
Let $x \in \tT(U)$. The following are equivalent.
\begin{enumerate} \eroman
    \item The MFCE-relation ${\Sis(x)}$  is tangible.

    \item For every son $z$ of $x$ in $U$ the element
    $[z]_{\Sis(x)}$ is tangible in $U/ \Sis(x)$ (and hence
    isolated).

    \item If  $u,v,w,p$ are elements of $U$ with $ evpx = ewpx = epx,$
    $uvpx \in \tT(U)$, $wpx \in \tT(U)$, then $uwpx \in \tT(U).$
\end{enumerate}
\end{thm}

Of course $ \Is(x) \subset  \Sis(x)$. We now exhibit a good case,
where these two MFCE-relations coincide. This means that for every
son $z$ of $x$ the element $[z]_{\Is(x)}$ is either ghost or
tangible in $U/ \Is(x)$.

\begin{thm}\label{thm:7.18}
Let $x \in \tT(U)$ and $a := ex.$ Assume that  the submonoid
$\tT(U)e$ of $M$ admits the cancellation hypothesis:
for any $b,c,d \in \tT(U)e$
$$abd = acd \dss \Rightarrow ab= ac.$$
(N.B. This certainly holds if $\tT(U)e$  is cancellative.) Then
$$ \Is(x) \ds = \Sis(x).$$
\end{thm}

\begin{proof}
Let $\gm = (vpx,u, wpx)$ be an elementary $\Sis(x)$-path
with $u,v,w,p \in U$, $vpx, wpx \in \tT(U)$, $ evpx = ewpx = epx.$ Then, by the  cancellation hypothesis, $evx = ewx = ex$.
Thus $\tlgm:= (vx,up,wx)$ is an $\Is(x)$-path with the same pair
of nodes as $\gm.$ It follows that, given an $\Sis(x)$-path
 form $z \in U$ to $w \in U,$ there also exists an $\Is(x)$-path
 from $z$ to $w.$ Thus $\Sis(x) \subset \Is(x)$, which
trivially implies $\Sis(x) = \Is(x)$.
\end{proof}

\end{document}